\documentclass{amsart}

\usepackage[leqno]{amsmath}
\usepackage{amsthm,latexsym,amsfonts,amssymb}
\usepackage[all]{xy} \SelectTips{eu}{}
\usepackage[hidelinks]{hyperref}
\usepackage{extarrows}
\usepackage{amscd}
\usepackage{xcolor}
\usepackage{mathrsfs}
\usepackage{tikz}
\usepackage{graphicx}

\newcommand{\numberseries}{\bfseries}   
\newlength{\thmtopspace}                
\newlength{\thmbotspace}                
\newlength{\thmheadspace}               
\newlength{\thmindent}                  
\newtheoremstyle{bfupright head,slanted body}
                {\thmtopspace}{\thmbotspace}
                {\slshape}{\thmindent}{\bfseries}{.}{\thmheadspace}
                {{\numberseries \thmnumber{#2\;}}\thmnote{#3}}

\newtheoremstyle{fixed bf head,slanted body}
                {\thmtopspace}{\thmbotspace}{\slshape}
                {\thmindent}{\bfseries}{.}{\thmheadspace}
                {{\numberseries \thmnumber{#2\;}}\thmname{#1}\thmnote{ (#3)}}

\newtheoremstyle{fixed bf head,upright body}
                {\thmtopspace}{\thmbotspace}{\upshape}
                {\thmindent}{\bfseries}{.}{\thmheadspace}
                {{\numberseries \thmnumber{#2\;}}\thmname{#1}\thmnote{ (#3)}}

\newtheoremstyle{numbered paragraph}
                {\thmtopspace}{\thmbotspace}{\upshape}
                {\thmindent}{\upshape}{}{\thmheadspace}
                {{\numberseries \thmnumber{#2.}}}

\theoremstyle{bfupright head,slanted body}
\newtheorem{res}{}[section]
\newtheorem*{res*}{}
\theoremstyle{fixed bf head,slanted body}
\newtheorem{thm}[res]{Theorem}          \newtheorem*{thm*}{Theorem}
\newtheorem{prp}[res]{Proposition}      \newtheorem*{prp*}{Proposition}
\newtheorem{cor}[res]{Corollary}        \newtheorem*{cor*}{Corollary}
\newtheorem{lem}[res]{Lemma}            \newtheorem*{lem*}{Lemma}
\newtheorem{exm}[res]{Example}
\theoremstyle{fixed bf head,upright body}
\newtheorem{dfn}[res]{Definition}       \newtheorem*{dfn*}{Definition}
     \newtheorem*{con*}{Construction}
\newtheorem{rmk}[res]{Remark}           \newtheorem*{rmk*}{Remark}
             \newtheorem*{fct*}{Fact}

\theoremstyle{numbered paragraph}

\newlength{\thmlistleft}        
\newlength{\thmlistright}       
\newlength{\thmlistpartopsep}   
\newlength{\thmlisttopsep}      
\newlength{\thmlistparsep}      
\newlength{\thmlistitemsep}     

\newcounter{eqc}
  {\end{list}}%

\newcounter{prt}
  {\end{list}}%

\newcounter{rqm}
\newenvironment{rqm}{\begin{list}{\upshape (\arabic{rqm})}%
    {\usecounter{rqm}%
      \setlength{\leftmargin}{\thmlistleft}%
      \setlength{\labelwidth}{\thmlistleft}%
      \setlength{\rightmargin}{\thmlistright}%
      \setlength{\partopsep}{\thmlistpartopsep}%
      \setlength{\topsep}{\thmlisttopsep}%
      \setlength{\parsep}{\thmlistparsep}%
      \setlength{\itemsep}{\thmlistitemsep}}}%
  {\end{list}}%

\newenvironment{prf*}[1][Proof]{%
  \begin{proof}[\bf #1]
    \setcounter{equation}{0}
    }
  {\end{proof}
}

\newcommand{\pgref}[1]{\ref{#1}}

\newcommand{\prpref}[2][Proposition~]{#1\pgref{prp:#2}}
\newcommand{\lemref}[2][Lemma~]{#1\pgref{lem:#2}}

\newcommand{\dfnref}[2][Definition~]{#1\pgref{dfn:#2}}
\newcommand{\rmkref}[2][Remark~]{#1\pgref{rmk:#2}}
\renewcommand{\eqref}[1]{(\pgref{eq:#1})}

\def\urltilda{\kern -.15em\lower .7ex\hbox{\~{}}\kern .04em}

\newcommand{\deq}{\:=\:}

\newcommand{\lra}{\longrightarrow}

\newcommand{\xra}[2][]{\xrightarrow[#1]{\;#2\;}}

\newcommand{\dif}[2][]{{\partial}^{#2}_{#1}}

\newcommand{\Cy}[2][]{\operatorname{Z}_{#1}(#2)}
\newcommand{\Co}[2][]{\operatorname{C}_{#1}(#2)}
\renewcommand{\H}[2][]{\operatorname{H}_{#1}(#2)}

\newcommand{\Hom}[3]{\operatorname{Hom}_{#1}(#2,#3)}

\newcommand{\ra}{\rightarrow}

\def\Q{\mathcal{Q}}
\def\I{\mathcal{I}}
\def\modcat{\mathrm{mod}}
\def\dashed{\text{-}}

\makeatletter
\def\@nobreak@#1{\mathchoice%
  {\nobreakdef@\displaystyle\f@size{#1}}%
  {\nobreakdef@\nobreakstyle\tf@size{\firstchoice@false #1}}%
  {\nobreakdef@\nobreakstyle\sf@size{\firstchoice@false #1}}%
  {\nobreakdef@\nobreakstyle\ssf@size{\firstchoice@false #1}}%
  \check@mathfonts}%
\def\nobreakdef@#1#2#3{\hbox{{%
                    \everymath{#1}%
                    \let\f@size#2\selectfont%
                    #3}}}%
\makeatother

\numberwithin{equation}{res}

\begin{document}

\title
{Cotorsion pairs in comma categories}

\author[Yuan Yuan]{Yuan Yuan}

\address{Department of Applied Mathematics, Lanzhou University of
  Technology, Lanzhou 730050, China}

\email{yynldxz@163.com}

\author[Jian He]{Jian He$^{\ast}$}

\address{Department of Applied Mathematics, Lanzhou University of
  Technology, Lanzhou 730050, China}

\email{jianhe30@163.com}

\author[Dejun Wu]{Dejun Wu}

\address{Department of Applied Mathematics, Lanzhou University of
  Technology, Lanzhou 730050, China}

\email{wudj@lut.edu.cn}

\thanks{$^{\ast}$Corresponding author. Jian He is supported by the National Natural Science Foundation of China (Grant No. 12171230) and Youth Science and Technology Foundation of Gansu Provincial (Grant No. 23JRRA825). ~D. Wu is supported by the National Natural Science Foundation of China (Grant No. 12261056).}

\date{\today\\\indent 2020\emph{\ Mathematics Subject Classification.} 18A25, 16E30, 18G25}


\begin{abstract}
  Let $\mathcal{A}$ and $\mathcal{B}$ be abelian categories with enough projective and injective objects, and $T:\mathcal{A}\rightarrow\mathcal{B}$ a left exact additive functor. Then one has a comma category $(\mathcal{B}\hspace{-0.1cm}\downarrow\hspace{-0.1cm}T)$. It is shown that If $T:\mathcal{A}\rightarrow\mathcal{B}$ is $\mathcal{X}$-exact, then $(^{\bot}\mathcal{X},~\mathcal{X})$ is a (hereditary) cotorsion pair in $\mathcal{A}$ and $(^{\bot}\mathcal{Y},~\mathcal{Y})$) is a (hereditary) cotorsion pair in $\mathcal{B}$ if and only if $((\begin{smallmatrix}^{\bot}\mathcal{X}\\^{\bot}\mathcal{Y}\end{smallmatrix}),\langle{\bf h}(\mathcal{X},~\mathcal{Y})\rangle)$ is a (hereditary) cotorsion pair in $(\mathcal{B}\downarrow T)$ and $\mathcal{X}$ and $\mathcal{Y}$ are closed under extensions. Furthermore, we characterize
  when special preenveloping classes in abelian categories $\mathcal{A}$ and $\mathcal{B}$ can induce special preenveloping classes in $(\mathcal{B}\downarrow T)$.
\end{abstract}

\maketitle

\thispagestyle{empty}


\section{Introduction}

\noindent Let $\mathcal{A}$ and $\mathcal{B}$ be abelian categories, and $T:\mathcal{A}\rightarrow\mathcal{B}$ be a left exact additive functor and $G:\mathcal{B}\rightarrow\mathcal{A}$ a right exact additive functor. It follows from Fossum, Griffith and Reiten \cite{MR389981} and Marmaridis \cite{MR709023} that there are two abelian comma categories $(\mathcal{B}\hspace{-0.05cm}\downarrow\hspace{-0.05cm}T)$ and $(G\hspace{-0.05cm}\downarrow\hspace{-0.05cm}\mathcal{A})$. Notice that module categories of upper triangular matrix rings are comma categories. For each recollement of triangulated categories, it follows from Chen and Le \cite{MR4430943} that there is an epivalence (provided that it is full and essentially surjective, and detects isomorphisms between objects; see Gabriel and Ro\u{\i}ter \cite{MR1239447}) between the middle category and the comma category associated with a triangle functor from the category on the right to the category on the left. Hu and Zhu \cite{MR4412784} characterize when special precovering classes in abelian categories $\mathcal{A}$ and $\mathcal{B}$ can induce special precovering classes in comma category $(G\hspace{-0.05cm}\downarrow\hspace{-0.05cm}\mathcal{A})$.

In this paper, we study the property of the extension closure of some classes of objects in $(\mathcal{B}\hspace{-0.1cm}\downarrow\hspace{-0.1cm}T)$, the exactness of the functor $\textbf{h}$ and the detailed description of orthogonal classes of a given class ${\bf h}(\mathcal{X},\mathcal{Y})$ in $(\mathcal{B}\hspace{-0.1cm}\downarrow\hspace{-0.1cm}T)$. Moreover, we characterize when complete hereditary cotorsion pairs in abelian categories $\mathcal{A}$ and $\mathcal{B}$ can induce complete hereditary cotorsion pairs in $(\mathcal{B}\hspace{-0.05cm}\downarrow\hspace{-0.05cm}T)$. Furthermore, we characterize when special preenveloping classes in abelian categories $\mathcal{A}$ and $\mathcal{B}$ can induce special preenveloping classes in $(\mathcal{B}\hspace{-0.05cm}\downarrow\hspace{-0.05cm}T)$.


\section{Notation and Terminology}

\noindent Throughout this paper, $\mathcal{A}$ and $\mathcal{B}$ are abelian categories. Let $\mathcal{X}$ be a subclass of $\mathcal{A}$. Set
$$^{\perp}\mathcal{X}=\{M~\in~\mathcal{A}~|~\textrm{Ext}^{1}_{\mathcal{A}}(M, X)=0~\text{for~every}~X~\in~\mathcal{X}\},$$
$$\mathcal{X}^{\perp}=\{ M~\in~\mathcal{A}~|~\textrm{Ext}^{1}_{\mathcal{A}}(X, M)=0~\text{for~every}~X~\in~\mathcal{X}\}.$$
Let $\emph{f}:X\rightarrow M$ be a morphism. It is called $\emph{special}~\mathcal{X}$-$\emph{precover}$ of $M$ if $f$ is surjective, $X\in\mathcal{X}$ and ker$\emph{f}\in \mathcal{X}^{\perp}$. Dually, a morphism $\emph{g}:M\rightarrow Y$ is called $\emph{special}~\mathcal{Y}$-$\emph{preenvelop}$ of $M$ if $\emph{g}$ is injective, $Y\in\mathcal{Y}$ and coker$\emph{g}\in\,^{\perp}\mathcal{Y}$. The class $\mathcal{X}$ is called $\emph{special-precovering}$ (resp. $\emph{special-preevenloping}$) if every object in $\mathcal{A}$ has $\emph{special}~\mathcal{X}$-$\emph{precover}$ (resp. $\emph{special}~\mathcal{X}$-$\emph{preenvelop}$).

Let $\mathcal{X}$ and $\mathcal{Y}$ be a subclass of $\mathcal{A}$. Recall from \cite{MR565595} that a pair $(\mathcal{X},~\mathcal{Y})$ is called a cotorsion pair if $\mathcal{X}=\,^{\bot}\mathcal{Y},~\mathcal{Y}=\mathcal{X}^{\bot}$.
A cotorsion pair $(\mathcal{X},~\mathcal{Y})$ is said to be \emph{hereditary} \cite{MR2355778} if $\textrm{Ext}^{i}_{\mathcal{A}}(X, M)=0$ for all $i~\geq~1$, $X\in\mathcal{X}$ and $Y\in\mathcal{Y}$, equivalently, if whenever $0\rightarrow X_{1}\rightarrow X_{2} \rightarrow X_{3} \rightarrow 0$ is exact  with $X_{2},~X_{3}\in \mathcal{X}$, then $X_{1}$ is also in $\mathcal{X}$, or equivalently, if whenever $0\rightarrow Y_{1}\rightarrow Y_{2}\rightarrow Y_{3}\rightarrow 0$ is exact with $Y_{1},~Y_{2}\in \mathcal{Y}$, then $Y_{3}$ is also in $\mathcal{Y}$.
A cotorsion pair $(\mathcal{X},~\mathcal{Y})$ is said to be \emph{complete} \cite{MR2985654} if $\mathcal{X}$ is special precovering and $\mathcal{Y}$ is special preenveloping. Moreover, if $\mathcal{A}$ has enough projective objects, the condition that $(\mathcal{X},~\mathcal{Y})$ is complete is equivalent to that $\mathcal{Y}$ is special preenveloping. Similarly, if $\mathcal{A}$ has enough injective objects, the condition that $(\mathcal{X},~\mathcal{Y})$ is complete is equivalent to that $\mathcal{X}$ is special precovering.
Let $\mathcal{X}$ be a subclass of $\mathcal{A}$. The functor $T:\mathcal{A}\rightarrow\mathcal{B}$ is called $\mathcal{X}$-$\emph{exact}$ if $T$ preserves  exactness of exact sequence $0\rightarrow X\rightarrow A_{1}\rightarrow A_{2}\rightarrow0$ in $\mathcal{A}$ with $X\in\mathcal{X}$.

Let $\mathcal{A}$ be an abelian category and $X$ a complex in $\mathcal{A}$. With homological grading, $X$ has the form
$$\cdots \lra X_{i+1} \xra{\dif[i+1]{X}} X_i \xra{\dif[i]{X}} X_{i-1} \lra \cdots.$$
We use the notations $\Cy[i]{X}$ for the kernel of differential $\dif[i]{X}$ and $\Co[i]{X}$ for the cokernel of the
differential $\dif[i+1]{X}$. A complex $X$ is called \emph{acyclic} if the homology complex $\H{X}$ is the
zero complex. A complex $\textbf{I}$ of injective objects is called $\emph{totally acyclic}$ if it is acyclic and $\textrm{Hom}_{\mathcal{A}}(E,~\textbf{I})$ is acyclic for any injective object $E$ in $\mathcal{A}$. An object $G$ in $\mathcal{A}$ is called $\emph{Gorenstein injective}$ if one has $G\cong\textrm{Z}_{0}(\textbf{I})$ for some totally acyclic complex $\textbf{I}$ of injective objects in $\mathcal{A}$. In what follows, we denote by $\mathcal{GI}_{\mathcal{A}}$ the subcategory of $\mathcal{A}$ consisting of Gorenstein injective objects.

Let $\mathcal{X}$ be a class of objects in an abelian category $\mathcal{A}$. $\mathcal{X}$ is said to be \emph{closed under extensions} if whenever $0\rightarrow X_{1}\rightarrow X_{2}\rightarrow X_{3}\rightarrow 0$ is an exact sequence in $\mathcal{A}$ with $X_{1},~X_{3}\in\mathcal{X}$, then $X_{2}\in\mathcal{X}$.

Recall the following definition of comma categories from \cite{MR389981,MR709023}.

\begin{dfn}
Let $\mathcal{A}$ and $\mathcal{B}$ be abelian categories, and $T:\mathcal{A}\rightarrow\mathcal{B}$ a left exact additive functor. The comma category, denoted by $(\mathcal{B}\hspace{-0.1cm}\downarrow\hspace{-0.1cm}T)$, is defined as follows. The objects of the category are triplets $\left(\begin{smallmatrix}A\\B\end{smallmatrix}\right)_{\varphi}$
with $A\in\mathcal{A}$, $B\in\mathcal{B}$ and $\varphi\in\textrm{Hom}_{\mathcal{B}}(B,~TA)$; and the morphisms of the category are $\left(\begin{smallmatrix}\alpha\\ \beta\end{smallmatrix}\right)_{\varphi}$
in $\mathcal{A}\times\mathcal{B}$ such that the following diagram
\[
\begin{CD}
B@>{\rm}^{\beta}>{\rm}>B^{'} \\
^{\varphi}@VVV ^{\varphi^{'}}@VVV \\
TA @>{\rm}^{T\alpha}>{\rm}>TA^{'}
\end{CD}
\]
is commutative.
\end{dfn}

\section{Homological behavior of the functor h}

\noindent Firstly, recall the following definition from \cite{MR389981,MR709023}.

\begin{dfn}
Let $T:\mathcal{A}\rightarrow\mathcal{B}$ be a left exact additive functor. Then we have the following functor:
 $\textbf{h}:\mathcal{A}\times\mathcal{B}\rightarrow(\mathcal{B}\downarrow T)$ via $\textbf{h}(A,~B)=\left(\begin{smallmatrix}A\\B\oplus TA\end{smallmatrix}\right)$
and $\textbf{h}(\alpha,~\beta)=\left(\begin{smallmatrix}\alpha \\ \beta\oplus T\alpha \end{smallmatrix}\right)$
, where $(A, B)$ is an object in $\mathcal{A}\times\mathcal{B}$ and $(\alpha, \beta)$ is a morphism in $\mathcal{A}\times\mathcal{B}$.
\end{dfn}

\begin{rmk}\label{rmk:functorH}
\begin{rqm}
  \item $\textbf{h}(A,~B)=\textbf{h}(A,~0)\oplus\textbf{h}(0,~B)$ with $A\in\mathcal{A}$, $B\in\mathcal{B}$. Moreover, $\textbf{h}$ perserves injective objects if $\mathcal{A}$ and $\mathcal{B}$ have enough injective objects.
  \item Define $\textbf{q}:(\mathcal{B}\downarrow T)\rightarrow\mathcal{A}\times\mathcal{B}$ via $\textbf{q} \left(\begin{smallmatrix}A\\B\end{smallmatrix}\right)=(A,~B)$
and $\bf{q}
 \left(\begin{smallmatrix}\alpha\\ \beta\end{smallmatrix}\right)=(\alpha,~\beta)$
, where $\left(\begin{smallmatrix}A\\B\end{smallmatrix}\right)$ is an object in $(\mathcal{B}\downarrow T)$ and $\left(\begin{smallmatrix}\alpha\\ \beta\end{smallmatrix}\right)$ is a morphism in $(\mathcal{B}\downarrow T)$, then $\textbf{q}$ is a left adjoint of $\bf{h}$. Hence $\textbf{h}$ is a left exact functor.
\end{rqm}
\end{rmk}

For convenience, we set
$$\mathfrak{D}^{\mathcal{X}}_{\mathcal{Y}}=\left\{\left(\hspace{-0.15cm}
         \begin{array}{c}
           X \\
           Y \\
         \end{array}\hspace{-0.15cm}
       \right)_{\varphi}\in(\mathcal{B}\downarrow T)\:\left|\:X\in\mathcal{X},~\varphi~\rm{is}~\rm{epimorphism}~
\rm{and}~\rm{ker}\varphi \in\mathcal{Y}\right.\right\}$$

\begin{equation*}
    \langle{\bf h}(\mathcal {X},\mathcal {Y})\rangle \deq %
    \left\{\left(\hspace{-0.15cm}
         \begin{array}{c}
           X \\
           Y \\
         \end{array}\hspace{-0.15cm}
       \right)
      \:\left|\:
        \begin{gathered}
          \text{$0\to \left(\hspace{-0.15cm}
         \begin{array}{c}
           X' \\
           Y' \\
         \end{array}\hspace{-0.15cm}
       \right)\to \left(\hspace{-0.15cm}
         \begin{array}{c}
           X \\
           Y \\
         \end{array}\hspace{-0.15cm}
       \right)\to \left(\hspace{-0.15cm}
         \begin{array}{c}
           X'' \\
           Y''\\
         \end{array}\hspace{-0.15cm}
       \right)\to 0$ is exact}\\[-3pt]
          \text{with $\left(\hspace{-0.15cm}
         \begin{array}{c}
           X' \\
           Y' \\
         \end{array}\hspace{-0.15cm}
       \right)$, $\left(\hspace{-0.15cm}
         \begin{array}{c}
           X'' \\
           Y'' \\
         \end{array}\hspace{-0.15cm}
       \right)\in \langle{\bf h}(\mathcal {X},\mathcal {Y})\rangle$}
        \end{gathered}
      \right.
    \right\}
  \end{equation*}

\begin{prp}\label{prp:hxyD}
Let $\mathcal{A}$ and $\mathcal{B}$ are abelian categories, $\mathcal{X}$ and $\mathcal{Y}$ subclasses of $\mathcal{A}$ and $\mathcal{B}$, respectively.
Considering the following conditions:
\begin{rqm}
  \item $\mathcal{X}$ and $\mathcal{Y}$ are closed under extensions in $\mathcal{A}$ and $\mathcal{B}$, respectively.
  \item $\langle{\bf h}(\mathcal {X},\mathcal {Y})\rangle=\mathfrak{D}^{\mathcal{X}}_{\mathcal{Y}}$ is the smallest subclass of $(\mathcal{B}\downarrow T)$ containing $\langle{\bf h}(\mathcal {X},\mathcal {Y})\rangle$ and closed under extensions.
  \item $\mathfrak{D}^{\mathcal{X}}_{\mathcal{Y}}$ is closed under extensions.
\end{rqm}
Then one has $(1)\Rightarrow(2)\Rightarrow(3)$. The converse holds if $T:\mathcal{A}\rightarrow\mathcal{B}$ is $\mathcal{X}$-exact.
\end{prp}
\begin{prf*}
$(1)\Rightarrow(2)$ Let $\left(\begin{smallmatrix}X\\Y\end{smallmatrix}\right)_{\varphi}\in\langle{\bf h}(\mathcal {X},\mathcal {Y})\rangle$.
There is an exact sequence by definition
$$0\rightarrow
\left(\hspace{-0.15cm}\begin{array}{c} X'\\Y'\oplus TX' \\\end{array}\hspace{-0.15cm}\right)\rightarrow
\left(\hspace{-0.15cm}\begin{array}{c} X\\Y\\\end{array}\hspace{-0.15cm}\right)_\varphi\rightarrow
\left(\hspace{-0.15cm}\begin{array}{c} X''\\Y''\oplus TX'' \\\end{array}\hspace{-0.15cm}\right)\rightarrow0,$$
which yields an exact sequence $0\rightarrow X^{'}\rightarrow X\rightarrow X^{''}\rightarrow0$ in $\mathcal{A}$. Consequently, one has the following commutative diagram with exact rows:
$$\xymatrix{0\ar[r]&Y'\oplus TX'\ar[r]\ar[d]&Y\ar[r]\ar[d]^\varphi&Y''\oplus TX''\ar[r]\ar[d]&0\\
    0\ar[r]&TX'\ar[r]&TX\ar[r]&TX''.}$$
We can get that $TX\rightarrow TX^{''}$ is an epimorphism. It follows from Five Lemma that $\varphi$ is a epimorphism, and  $X\in\mathcal{X}$ since $\mathcal{X}$ is closed under extensions. By the Snake Lemma, there is an exact sequence $0\rightarrow Y'\rightarrow\textrm{ker}\varphi\rightarrow Y^{''}\rightarrow0$ which implies $\textrm{ker}\varphi\in\mathcal{Y}$ since $\mathcal{Y}$ is closed under extensions.

Conversely, assume that $\left(\begin{smallmatrix}X\\Y\end{smallmatrix}\right)_{\varphi}\in\mathfrak{D}^{\mathcal{X}}_{\mathcal{Y}}$.
Then the exact sequence $$0\rightarrow
\left(\hspace{-0.15cm}\begin{array}{c} 0\\ \rm{ker}\varphi\\\end{array}\hspace{-0.15cm}\right)\rightarrow
\left(\hspace{-0.15cm}\begin{array}{c} X\\Y\\\end{array}\hspace{-0.15cm}\right)_\varphi\rightarrow
\left(\hspace{-0.15cm}\begin{array}{c} X\\TX\\\end{array}\hspace{-0.15cm}\right)\rightarrow0$$ yields that $\left(\begin{smallmatrix}X\\Y\end{smallmatrix}\right)_{\varphi}\in\langle{\bf h}(\mathcal {X},\mathcal {Y})\rangle$.
By the proof above, one can get that $\langle{\bf h}(\mathcal {X},\mathcal {Y})\rangle$ is closed under extensions. So $\langle{\bf h}(\mathcal {X},\mathcal {Y})\rangle$ is the smallest subclass of $(\mathcal{B}\downarrow T)$ containing ${\bf h}(\mathcal {X},\mathcal {Y})$ and closed under extensions.

$(2)\Rightarrow(3)$ It is trivial.

$(3)\Rightarrow(1)$ Let $0\rightarrow A'\rightarrow A\rightarrow A''\rightarrow0$ be an exact sequence in $\mathcal{A}$ with $A',A''\in{\mathcal{X}}$. Then there is an exact sequence$0\rightarrow TA'\rightarrow TA\rightarrow TA''\rightarrow0$ in $\mathcal{B}$ with $T$ is $\mathcal{X}$-exact by hypothesis.
Now there is an exact sequence $0\rightarrow\left(\begin{smallmatrix}A'\\TA' \end{smallmatrix}\right)\rightarrow \left(\begin{smallmatrix}A\\TA\end{smallmatrix}\right)\rightarrow\left(\begin{smallmatrix} A''\\TA''\end{smallmatrix}\right)\rightarrow0 $  in $(\mathcal{B}\downarrow T)$ with $\left(\begin{smallmatrix}A'\\TA'\end{smallmatrix}\right)$, $\left(\begin{smallmatrix}A''\\TA''\end{smallmatrix}\right)\in \mathfrak{D}^{\mathcal{X}}_{\mathcal{Y}}$ by hypothesis. It follows that $\left(\begin{smallmatrix}A\\TA\end{smallmatrix}\right)\in  \mathfrak{D}^{\mathcal{X}}_{\mathcal{Y}}$, which implies that $A\in{\mathcal{X}}$. Moreover, let $0\rightarrow B'\rightarrow B\rightarrow B''\rightarrow0$ be an exact sequence in $\mathcal{B}$ with $B',B''\in\mathcal{Y}$. Accordingly, we get an exact sequence $0\rightarrow\left(\begin{smallmatrix}0\\B' \end{smallmatrix}\right)\rightarrow \left(\begin{smallmatrix}0\\B\end{smallmatrix}\right)\rightarrow\left(\begin{smallmatrix} 0\\B''\end{smallmatrix}\right)\rightarrow0$   in $(\mathcal{B}\downarrow T)$ with $\left(\begin{smallmatrix}0\\B'\end{smallmatrix}\right)$, $\left(\begin{smallmatrix}0\\B''\end{smallmatrix}\right)\in{ \mathfrak{D}^{\mathcal{X}}_{\mathcal{Y}}}$. Thus $\left(\begin{smallmatrix}0\\B\end{smallmatrix}\right)\in{ \mathfrak{D}^{\mathcal{X}}_{\mathcal{Y}}}$ by hypothesis, and hence we get that $B\in{\mathcal{Y}}$. So $\mathcal{Y}$ is closed under extensions, as desired.
\end{prf*}

\begin{prp}
Let $T:\mathcal{A}\rightarrow\mathcal{B}$ be a left exact functor. The following are true:
 \begin{enumerate}
\item $T$ is $\mathcal {X}$-exact if and only if two exact sequences $0\rightarrow X\rightarrow A\rightarrow A''\rightarrow0$ in $\mathcal{A}$ and $0\rightarrow B'\rightarrow B\rightarrow B''\rightarrow0$ in $\mathcal{B}$ with $X\in{\mathcal{X}}$ induce an exact sequence $0\rightarrow{\bf h}(X,~B')\rightarrow{\bf h}(A,~B)\rightarrow{\bf h}(A'',~B'')\rightarrow0 $ in $(\mathcal{B}\downarrow T)$.
 \item $T$ is exact if and only if ${\bf h}:\mathcal{A}\times\mathcal{B}\rightarrow(\mathcal{B}\downarrow T)$ is exact.
 \end{enumerate}
\end{prp}
\begin{proof}  (1) Assume that $T:\mathcal{A}\rightarrow\mathcal{B}$ is $\mathcal{X}$-exact. Let $0\rightarrow X\rightarrow A\rightarrow A''\rightarrow0$ be an exact sequence in $\mathcal{A}$ with $X\in\mathcal{X}$ and $0\rightarrow B'\rightarrow B\rightarrow B''\rightarrow0$ in $\mathcal{B}$ be an exact sequence. Accordingly, there is an exact sequence $0\rightarrow TX\rightarrow TA\rightarrow TA''\rightarrow0$ in $\mathcal{B}$. So the sequence $0\rightarrow{\bf h}(X,~B')\rightarrow{\bf h}(A,~B)\rightarrow{\bf h}(A'',~B'')\rightarrow0 $ is exact in $(\mathcal{B}\downarrow T)$.

Conversely, let $0\rightarrow X\rightarrow A\rightarrow A''\rightarrow0$ be an exact sequence in $\mathcal{A}$ with $X\in\mathcal {X}$. Note that $0\rightarrow{\bf h}(X,~0)\rightarrow{\bf h}(A,~0)\rightarrow{\bf h}(A'',~0)\rightarrow0 $ is an exact sequence in $(\mathcal{B}\downarrow T)$ by hypothesis. So $0\rightarrow TX\rightarrow TA\rightarrow TA''\rightarrow0$ is an exact sequence in $\mathcal{B}$, $T$ is $X$-exact. (2) is a direct consequence of (1).
\end{proof}

\begin{prp}\label{prp:Xexact}
Let $T:\mathcal{A}\rightarrow\mathcal{B}$ be $\mathcal {X}$-exact. Then  $^{\bot}\langle{\bf h}(\mathcal {X},\mathcal {Y})\rangle=(\begin{smallmatrix}^{\bot}\mathcal{X}\\^{\bot}\mathcal{Y}\end{smallmatrix})$ holds in the category $(\mathcal{
B}\downarrow T)$.
\end{prp}
\begin{proof}
It follows from \rmkref{functorH} that
$$^{\bot}\langle{\bf h}(\mathcal{X},~\mathcal{Y})\rangle=^{\bot}{\bf h}(\mathcal{X},~\mathcal{Y})=^{\bot}{\bf h}(\mathcal{X},~0)\cap{^{\bot}\bf h}(0,~\mathcal{Y}).$$
Firstly, we claim that $(\begin{smallmatrix}^{\bot}\mathcal{X}\\^{\bot}\mathcal{Y} \end{smallmatrix}) \subseteq{^{\bot}\bf h}(\mathcal {X},\mathcal {Y})$.
Let $\left(\begin{smallmatrix}A\\B \end{smallmatrix}\right)\in (\begin{smallmatrix}^{\bot}\mathcal{X}\\ ^{\bot}\mathcal {Y} \end{smallmatrix})$. It is sufficient to show the following exact sequences $$\zeta:0\rightarrow\left(\hspace{-0.15cm}\begin{array}{c} X\\TX\\\end{array}\hspace{-0.15cm}\right)
\stackrel{{\tiny\left(\begin{smallmatrix}\hspace{-0.1cm}m\hspace{-0.1cm}\\
\hspace{-0.1cm}n\hspace{-0.1cm}\end{smallmatrix}\right)}}{\rightarrow} \left(\hspace{-0.15cm}\begin{array}{c} M\\N\\\end{array}\hspace{-0.15cm}\right){\rightarrow}
\left(\hspace{-0.15cm}\begin{array}{c} A\\B\\\end{array}\hspace{-0.15cm}\right)\rightarrow0~\rm{and}~ \xi:0\rightarrow\left(\hspace{-0.15cm}\begin{array}{c} 0\\Y\\\end{array}\hspace{-0.15cm}\right)
\stackrel{{\tiny\left(\begin{smallmatrix}\hspace{-0.05cm}0\hspace{-0.05cm}\\
\hspace{-0.1cm}f\hspace{-0.1cm}\end{smallmatrix}\right)}}{\rightarrow} \left(\hspace{-0.15cm}\begin{array}{c} A\\C\\\end{array}\hspace{-0.15cm}\right){\rightarrow}
\left(\hspace{-0.15cm}\begin{array}{c} A\\B\\\end{array}\hspace{-0.15cm}\right)\rightarrow0$$
are split for any $X\in{\mathcal{X}}$ and $Y\in{\mathcal{Y}}$.
There is a split exact sequence $0\rightarrow X\stackrel{{\tiny{m}}}{\rightarrow}M\rightarrow A\rightarrow0$ in $\mathcal{A}$, thus existing $m':M\rightarrow X$ such that $m'm=1_{X}$.
It follows that $(Tm'\varphi)Tn=(Tm')Tm=T(m'm)=1_{TX}$, i.e.,   $(\begin{smallmatrix}m'\\Tm'\varphi\end{smallmatrix})\left(\begin{smallmatrix}m\\n\end{smallmatrix}\right)=1$  which implies that the sequence $\zeta$ is split, as desired. Similarly, it is easy to see that $\xi$ is split since the exact sequence $0\rightarrow Y\rightarrow C\rightarrow D\rightarrow0$ is split.

Next we claim that $^{\bot}\langle{\bf h}(\mathcal {X},\mathcal {Y})\rangle \subseteq (\begin{smallmatrix}^{\bot}\mathcal{X}\\ ^{\bot}\mathcal {Y} \end{smallmatrix})$. We only need to show that $^{\bot}{\bf h}(\mathcal {X},0)\subseteq\left(\begin{smallmatrix}^{\bot}\mathcal {X}\\ \mathcal{B} \end{smallmatrix}\right)$ and $^{\bot}{\bf h}(0,\mathcal {Y})\subseteq(\begin{smallmatrix}\mathcal{A}\\^{\bot}\mathcal{Y}\end{smallmatrix})$.
Let $\left(\begin{smallmatrix}A\\B\end{smallmatrix}\right)$ be an object an $^{\bot}{\bf h}(\mathcal {X},0)$. Assume that $\varepsilon:0\rightarrow X\rightarrow M\rightarrow A\rightarrow0$ is an exact sequence in $\mathcal{A}$ with $X\in\mathcal{X}$. It follows from $T$ is $\mathcal{X}-$exact that  $0\rightarrow TX\rightarrow TM\rightarrow TA\rightarrow0$ is exact in $\mathcal{B}$. Consider the following pullback diagram:
 $$\xymatrix{0\ar[r]&TX\ar[r]\ar[d]&N\ar[r]\ar[d]&B\ar[r]\ar@{=}[d]&0\\
    0\ar[r]&TX \ar[r]&TM\ar[r]&TA \ar[r]&0.}$$
There is an exact sequence $\varrho:0\rightarrow\left(\begin{smallmatrix}X\\TX \end{smallmatrix}\right){\rightarrow} \left(\begin{smallmatrix}M\\N\end{smallmatrix}\right)\rightarrow\left(\begin{smallmatrix}A\\B\end{smallmatrix}\right)\rightarrow0$ in $(\mathcal{B}\downarrow T)$. Hence the sequence $\varrho$ is split by hypothesis, so the sequence $\varepsilon$ is split, i.e. $A\in\,^{\bot}\mathcal{X}$. Consequently, $^{\bot}{\bf h}(\mathcal {X},0)\subseteq(\begin{smallmatrix}^{\bot}\mathcal {X}\\ \mathcal{B}\end{smallmatrix})$.

Let $\left(\begin{smallmatrix}A\\B\end{smallmatrix}\right)$ be an object in $^{\bot}{\bf h}(0,\mathcal {X})$ and $\sigma:0\rightarrow Y\rightarrow D\rightarrow B\rightarrow0$ is an exact sequence in $\mathcal{B}$ with $Y\in\mathcal{Y}$. Then there is an exact sequence $\theta:0\rightarrow\left(\begin{smallmatrix}0\\Y \end{smallmatrix}\right){\rightarrow} \left(\begin{smallmatrix}A\\D\end{smallmatrix}\right)\rightarrow\left(\begin{smallmatrix}A\\B\end{smallmatrix}\right)\rightarrow0$ and $\theta$ is split by hypothesis. So the sequence $\sigma$ is split, i.e., $B\in\,^{\bot}\mathcal{Y}$. Consequently, $^{\bot}{\bf h}(0,~\mathcal {Y})\subseteq(\begin{smallmatrix}\mathcal{A}\\ ^{\bot}\mathcal{Y}\end{smallmatrix})$.
\end{proof}

\begin{lem}\label{lem:projective}
 If $\mathcal{B}$ has enough projective objects, then $\langle{\bf h}({\mathcal{A}},~{\mathcal {B}})\rangle={\left(\begin{smallmatrix}0\\ \mathcal {P}\end{smallmatrix}\right)^{\bot}}$, where $\mathcal{P}$ is the class of projective objects in $\mathcal{B}$.
\end{lem}
\begin{proof}
Let $\mathcal{B}$ have enough projective objects and $\mathcal{P}$ be the class of projective objects in $\mathcal{B}$. Assume that$\left(\begin{smallmatrix}A\\B\end{smallmatrix}\right)_\varphi\in{\left(\begin{smallmatrix}0\\ \mathcal{P} \end{smallmatrix}\right)^{\bot}}$ and there is a epimorphism $\sigma:P\rightarrow TA$ with $P$ is a projective object in $\mathcal{B}$. Then one has the following commutative diagram with exact rows:
$$\xymatrix{0\ar[r]&B \ar[r]^{\hspace{-0.3cm}\left(\begin{smallmatrix}1\\ 0\end{smallmatrix}\right)}\ar[d]^{\varphi}&B\oplus P\ar[r]^{\hspace{0.3cm}(0,1)}\ar[d]^{(\varphi,\sigma)}&P \ar[r]\ar[d]&0\\
0\ar[r]&T(A)\ar@{=}[r]&T(A)\ar[r]&0\ar[r] &0
.}$$

Clearly, we have an exact sequence $\xi:0\rightarrow\left(\begin{smallmatrix}A\\B\end{smallmatrix}\right){\rightarrow} \left(\begin{smallmatrix}A\\B\oplus P\end{smallmatrix}\right)\rightarrow\left(\begin{smallmatrix}0\\P\end{smallmatrix}\right)\rightarrow0$ in $(\mathcal{B}\downarrow T)$ and $\xi$ is split by hypothesis. It follows that there is a morphism $\left(\begin{smallmatrix}f\\g\end{smallmatrix}\right):P\rightarrow P\oplus B$ such that $(0,1)\left(\begin{smallmatrix}f\\g\end{smallmatrix}\right)=1_P$ and $(\varphi,~\sigma)\left(\begin{smallmatrix}f\\g\end{smallmatrix}\right)=0$ which implies $g=1_P$ and $\varphi f=-\sigma$. So $\varphi$ is a epimorphism, i.e.,  $\left(\begin{smallmatrix}A\\B\end{smallmatrix}\right)_\varphi\in\langle{\bf h}({\mathcal {A}},~{\mathcal {B}})\rangle$ by \prpref{hxyD}.

Conversely, let $\left(\begin{smallmatrix}A\\B\end{smallmatrix}\right)_\varphi\in\langle{\bf h}({\mathcal{A}},~{\mathcal{B}})\rangle$. Consider the exact sequence $$\sigma:0\rightarrow\left(\hspace{-0.15cm}\begin{array}{c} A\\B\\\end{array}\hspace{-0.15cm}\right)_\varphi
\overset{{{\tiny\left(\begin{smallmatrix}\hspace{-0.08cm}1\hspace{-0.1cm}\\ \hspace{-0.08cm}\pi\hspace{-0.1cm}\end{smallmatrix}\right)}}}\longrightarrow \left(\hspace{-0.15cm}\begin{array}{c} A\\M\\\end{array}\hspace{-0.15cm}\right)_{\phi}
\overset{{{\tiny\left(\begin{smallmatrix}\hspace{-0.08cm}0\hspace{-0.1cm}\\
\hspace{-0.08cm}p\hspace{-0.1cm}\end{smallmatrix}\right)}}
}\longrightarrow
\left(\hspace{-0.15cm}\begin{array}{c} 0\\P\\\end{array}\hspace{-0.15cm}\right)\rightarrow0$$
with $P$ projective. Since $P$ is projective, there  exists a map $\pi':M\rightarrow B$ such that $\pi'\pi=1_B$.
Hence we have $(\phi-\varphi \pi')\pi=0$. By the universal property of the cokernel, we get a map $\alpha:P\rightarrow TA$ such that $\phi-\varphi \pi'=\alpha p$. Since $\varphi$ is an epimorphism by \prpref{hxyD} and $P$ is projective, there exists a map $\beta:P\rightarrow B$ such that $\alpha=\varphi\beta$. Setting $q=\pi'+\beta p$, we have $$\varphi q=\varphi(\pi'+\beta p)=\varphi \pi'+\varphi\beta p=\varphi \pi'+\alpha p=\phi,$$
$$q\pi= (\pi'+\beta p)\pi=\pi'\pi+\beta p\pi =1_B.$$
Thus $\left(\begin{smallmatrix}\hspace{-0.08cm}1\hspace{-0.1cm}\\
\hspace{-0.08cm}q\hspace{-0.1cm}
\end{smallmatrix}\right)\left(\begin{smallmatrix}\hspace{-0.08cm}1\hspace{-0.1cm}\\ \hspace{-0.08cm}\pi\hspace{-0.1cm}\end{smallmatrix}\right)=1$, and so $\xi$ is split.
\end{proof}

\begin{prp}\label{prp:rightoth}
If $\mathcal{B}$ has enough projective objects, then $$\langle{\bf h}(\mathcal{X}^{\bot},~\mathcal {Y}^{\bot})\rangle={\left(\begin{smallmatrix}\mathcal {X}\\ \mathcal{Y} \end{smallmatrix}\right)^{\bot}}\cap{\left(\begin{smallmatrix}{0}\\ \mathcal{P}\end{smallmatrix}\right)^{\bot}},$$ where $\mathcal{P}$ is the class of projective objects in $\mathcal{B}$.
\end{prp}
\begin{proof}
Let $\mathcal{B}$ have enough projective objects and $\mathcal{P}$ be the class of projective objects in $\mathcal{B}$. At first, we claim that ${\left(\begin{smallmatrix}\mathcal{X}\\ \mathcal{Y} \end{smallmatrix}\right)^{\bot}}\bigcap{\left(\begin{smallmatrix}0\\ \mathcal{P} \end{smallmatrix}\right)^{\bot}}\subseteq\langle{\bf h}({\mathcal {X}^{\bot}},~{\mathcal{Y}^{\bot}})\rangle$.
Assume that $\left(\begin{smallmatrix}A\\B\end{smallmatrix}\right)_\varphi$ is an object in ${\left(\begin{smallmatrix}\mathcal{X}\\ \mathcal{Y} \end{smallmatrix}\right)^{\bot}}\bigcap{\left(\begin{smallmatrix}0\\ \mathcal{P} \end{smallmatrix}\right)^{\bot}}$. It follows from \prpref{hxyD} and \lemref{projective} that $\varphi$ is an epimorphism. Note that both ${^{\bot}\mathcal {X}}$ and ${^{\bot}\mathcal {Y}}$ are closed under extensions, and so it is sufficient to show $A\in{\mathcal {X}^{\bot}}$ and ker$\varphi\in{\mathcal{Y}^{\bot}}$ by \prpref{hxyD}.

Let $\zeta:0\rightarrow A\rightarrow M\rightarrow X\rightarrow0$ be an exact sequence in $\mathcal{A}$ with $X\in\mathcal{X}$. Consequently, there is an exact sequence $\xi:0\rightarrow\left(\begin{smallmatrix}A\\B\end{smallmatrix}\right){\rightarrow} \left(\begin{smallmatrix}M\\B\end{smallmatrix}\right)\rightarrow\left(\begin{smallmatrix}X\\0\end{smallmatrix}\right)\rightarrow0$ in $(\mathcal{B}\downarrow T)$. By hypothesis, the sequence $\xi$ is split since $\left(\begin{smallmatrix}X\\0\end{smallmatrix}\right)\in\left(\begin{smallmatrix}\mathcal {X}\\ \mathcal{Y} \end{smallmatrix}\right)$. So the sequence $\zeta$ is split and $A\in{\mathcal{X}^{\bot}}$, as desired.

Let $0\rightarrow\mbox{ker}\varphi{\rightarrow}M\stackrel{\tau}{\rightarrow}Y\rightarrow0$ be an exact sequence in $\mathcal{B}$ with $Y\in\mathcal{Y}$.
Then we have a pushout diagram

$$\xymatrix{&0\ar[d]&0\ar[d]\\
0\ar[r]&\rm{ker}\varphi\ar[r]\ar[d]^{k}&M\ar[r]^{\tau}\ar[d]^{k'}&Y\ar[r]\ar@{=}[d]&0\\
0\ar[r]&B\ar[r]^{\pi}\ar[d]^\varphi&L\ar[d]^{\varphi '}\ar[r]^{\tau '}&Y\ar[r]&0\\
&TA\ar@{=}[r]\ar[d]&TA\ar[d] &\\
&0&0&\\}$$
which induces an exact sequence
$$\xi':0\rightarrow\left(\hspace{-0.15cm}\begin{array}{c} A\\B\\\end{array}\hspace{-0.15cm}\right)_{\varphi}
\overset{\left(\begin{smallmatrix}\hspace{-0.08cm}1\hspace{-0.1cm}\\ \hspace{-0.08cm}\pi\hspace{-0.1cm}\end{smallmatrix}\right)}{\longrightarrow}
\left(\hspace{-0.15cm}\begin{array}{c} A\\L\\\end{array}\hspace{-0.15cm}\right)_{\widetilde{\varphi}}
\overset{\left(\begin{smallmatrix}\hspace{-0.08cm}0\hspace{-0.1cm}
\\\hspace{-0.08cm}{\tau'}\hspace{-0.1cm}\end{smallmatrix}\right)}{\longrightarrow}
\left(\hspace{-0.15cm}\begin{array}{c} 0\\Y\\\end{array}\hspace{-0.15cm}\right)\rightarrow0$$
in $(\mathcal{B}\hspace{-0.1cm}\downarrow \hspace{-0.1cm}T)$. Since $\xi'$ is split by hypothesis, there exists $\varsigma:Y\rightarrow L$ such that $\tau'\varsigma=1$ and $\varphi'\varsigma=0$. By the universal property of the kernel, there is a morphism ${g}:Y\rightarrow M$ such that $\varsigma=k'g $. Thus we have $\tau g=\tau'k'g=\tau'\varsigma=1$ which means the first row in the above diagram is split. So ker$\varphi\in{\mathcal{Y}^{\bot}}$, as required.

Nextly, we claim that $\langle{\bf h}({\mathcal{X}^{\bot}},~{\mathcal{Y}^{\bot}})\rangle\subseteq{\left(\begin{smallmatrix}\mathcal{X}\\ \mathcal{Y} \end{smallmatrix}\right)^{\bot}}\bigcap{\left(\begin{smallmatrix}0\\ \mathcal{P} \end{smallmatrix}\right)^{\bot}}$. By \rmkref{functorH} and \lemref{projective}, we only need to show that ${\bf h}(^{\bot}\mathcal {X},~0)\subseteq{\left(\begin{smallmatrix}\mathcal {X}\\ \mathcal{Y} \end{smallmatrix}\right)^{\bot}}$ and ${\bf h}(0,~^{\bot}\mathcal {Y})\subseteq{\left(\begin{smallmatrix}\mathcal {X}\\ \mathcal{Y} \end{smallmatrix}\right)^{\bot}}$.
Let $A\in{\mathcal{X}^{\bot}}$ and $\varepsilon:0\rightarrow\left(\begin{smallmatrix}A\\TA \end{smallmatrix}\right)\overset{\left(\begin{smallmatrix}\hspace{-0.08cm}m
\hspace{-0.1cm}\\ \hspace{-0.08cm}n\hspace{-0.1cm}\end{smallmatrix}\right)}{\longrightarrow}\left(\begin{smallmatrix}M\\N\end{smallmatrix}\right)_\phi{\rightarrow}
\left(\begin{smallmatrix}X\\Y\end{smallmatrix}\right)\rightarrow0$ be an exact sequence in $(\mathcal{B}\downarrow T)$ with $X\in{\mathcal{X}}$ and $Y\in{\mathcal{Y}}$. Note that the exact sequence $0\rightarrow A\rightarrow M\rightarrow X\rightarrow0$ in $\mathcal{A}$ is split. Thus there exists $m':M\rightarrow A$ such that $m'm=1_A$ and $(Tm')\phi n=(Tm')Tm=1_{TX}$. It follows that the sequence $\varepsilon$ is split. So ${\bf h}(^{\bot}\mathcal {X},~0)\subseteq{\left(\begin{smallmatrix}\mathcal {X}\\ \mathcal{Y} \end{smallmatrix}\right)^{\bot}}$. Similarly, one can show that $ {\bf h}(0,~^{\bot}\mathcal{Y})\subseteq{\left(\begin{smallmatrix}\mathcal {X}\\ \mathcal{Y} \end{smallmatrix}\right)^{\bot}}$.
\end{proof}

\begin{cor}
Let $\mathcal{A}$ and $\mathcal{B}$ be abelian categories with enough projective objects and enough injective objects. If $T:\mathcal{A}\rightarrow\mathcal{B}$ is an exact functor, then $\langle{\bf h}(\mathcal{A},~\mathcal{B})\rangle$ is a Frobenius category if and only if $\mathcal{A}$ and $\mathcal{B}$ are Frobenius and $T$ preserves injective objects.
\end{cor}
\begin{proof}
Let $\mathcal{P}$ (\textrm{resp.} $\mathcal{Q}$) be the class of projective objects in $\mathcal{A}$ (resp. $\mathcal{B}$) and $\mathcal{I}$ (\textrm{resp.} $\mathcal{J}$) the class of projective objects in $\mathcal{A}$ (resp. $\mathcal{B}$).

Then one has
$$\begin{array}{lll}{{\bf h}(\mathcal {A},~0)^{\bot}}\bigcap{\left(\begin{smallmatrix}0\\ \mathcal{P}  \end{smallmatrix}\right)^{\bot}}&={\left(\begin{smallmatrix}\mathcal {A}\\T\mathcal {A} \end{smallmatrix}\right)^{\bot}}\bigcap{\left(\begin{smallmatrix}0\\ \mathcal {P}\end{smallmatrix}\right)^{\bot}}&\\&=\langle{\bf h}({\mathcal{A}^{\bot}},~({T\mathcal{A})^{\bot}})\rangle&\\&=\langle{\bf h}(\mathcal{I},~({T\mathcal{A})^{\bot}})\rangle,\end{array}$$
$$\begin{array}{lll}{{\bf h}(0,~\mathcal{B})^{\bot}}\bigcap{\left(\begin{smallmatrix}0\\ \mathcal{P} \end{smallmatrix}\right)^{\bot}}&={\left(\begin{smallmatrix}0\\ \mathcal{B} \end{smallmatrix}\right)^{\bot}}\bigcap{}\left(\begin{smallmatrix}0\\ \mathcal{P} \end{smallmatrix}\right)^{\bot}&\\&=\langle{\bf h}({0^{\bot}},~{\mathcal{B}^{\bot}})\rangle&\\&=\langle{\bf h}(\mathcal{A},~\mathcal{J})\rangle.\end{array}$$

Thus one gets
$$\begin{array}{lll}{\langle{\bf h}(\mathcal{A},~\mathcal{B})\rangle^{\bot}}&={{\bf h}(\mathcal{A},~\mathcal{B})^{\bot}}={{\bf h}(\mathcal{A},~0)^{\bot}}\bigcap{{\bf h}(0,~\mathcal{B})^{\bot}}&\\&={{\bf h}(\mathcal{A},~0)^{\bot}}\bigcap{^{\bot}\left(\begin{smallmatrix}0\\ \mathcal {P} \end{smallmatrix}\right)}\bigcap{{\bf h}(0,~\mathcal{B})^{\bot}}&\\&=\left({{\bf h}(\mathcal {A},~0)^{\bot}}\bigcap{^{\bot}\left(\begin{smallmatrix}0\\ \mathcal{P} \end{smallmatrix}\right)}\right)\bigcap\left({{\bf h}(0,~\mathcal {B})^{\bot}}\bigcap{^{\bot}\left(\begin{smallmatrix}0\\ \mathcal{P}\end{smallmatrix}\right)}\right)\\&=\langle{\bf h}(\mathcal{I},~T\mathcal{A}^{\bot})\rangle\bigcap\langle{\bf h}({\mathcal{A}},~\mathcal{J})\rangle&\\&=\langle{\bf h}(\mathcal{I},~\mathcal{J})\rangle.\end{array}$$

Moreover, the class of projective objects in the exact category $\langle{\bf h}(\mathcal{A},~\mathcal{B})\rangle$ is $\left(\begin{smallmatrix}\mathcal {P}\\ \mathcal{Q}\end{smallmatrix}\right)\cap\langle{\bf h}(\mathcal{A},~\mathcal{B})\rangle$.
Therefore, $\langle{\bf h}(\mathcal{A},~\mathcal{B})\rangle$ is Frobenius if and only if $T(\mathcal{I})\subset\mathcal{Q}=\mathcal{J}$ and $\mathcal{I}=\mathcal{P}$, as desired.
\end{proof}

\section{Complete Cotorsion pairs}

\noindent In this section, we characterize when complete hereditary cotorsion pairs in abelian categories $\mathcal{A}$ and $\mathcal{B}$ can induce complete hereditary cotorsion pairs in $(\mathcal{B}\hspace{-0.05cm}\downarrow\hspace{-0.05cm}T)$.

\begin{lem}\label{orth5}
The following hold for a comma category $(\mathcal{B}\downarrow T)$:
 \begin{enumerate}
\item  $\left(\begin{smallmatrix}\mathcal{X}\\ \mathcal{Y}\end{smallmatrix}\right)$ is resolving in $(\mathcal{B}\downarrow T)$ if and only if $\mathcal {X}$ and $\mathcal{Y}$ are resolving in $\mathcal{A}$ and $\mathcal{B}$ respectively.
\item  If $T:\mathcal{A}\rightarrow\mathcal{B}$ is $\mathcal{X}$-exact and $\mathcal{X}$, $\mathcal{Y}$ are closed under extensions in $\mathcal{A}$ and $\mathcal{B}$, respectively, then $\langle{\bf h}(\mathcal {X},~\mathcal {Y})\rangle$ is coresolving in $(\mathcal{B}\downarrow T)$ if and only if $\mathcal {X}$ and $\mathcal{Y}$ are coresolving in $\mathcal{A}$ and $\mathcal{B}$, respectively.

   \end{enumerate}
\end{lem}
\begin{proof}  The proof of (1) is straightforward.

(2) Assume that $T$ is $\mathcal{X}$-exact. For the ``only if" part, we assume that $\langle{\bf h}(\mathcal {X},~\mathcal {Y})\rangle$ is coresolving in $(\mathcal{B}\downarrow T)$. Let $0\rightarrow X'\rightarrow X\rightarrow X''\rightarrow0$ be an exact sequence in $\mathcal{A}$ with $X'$ $\in$ $\mathcal{X}$. It follows that the sequence $0\rightarrow{\bf h}(X',~0)\rightarrow{\bf h}(X,~0)\rightarrow{\bf h}(X'',~0)\rightarrow0 $ is exact in $(\mathcal{B}\downarrow T)$. Thus we get that ${\bf h}(X'',~0)\in\langle{\bf h}(\mathcal{X},~\mathcal{Y})\rangle$ if and only if  ${\bf h}(X',~0)\in\langle{\bf h}(\mathcal{X},~\mathcal{Y})\rangle$ since$\langle{\bf h}(\mathcal {X},~\mathcal {Y})\rangle$ is coresolving. Hence $X\in{\mathcal{X}}$ if and only if $X''\in{\mathcal{X}}$. So $\mathcal{X}$ is coresolving in $\mathcal{A}$. Similarly, one can show that $\mathcal{Y}$ is coresolving in $\mathcal{B}$.

For the ``if" part, let  $\mathcal {X}$ and $\mathcal {Y}$  coresolving in $\mathcal{A}$ and $\mathcal{B}$, respectively. Then $\mathcal{X}$ and $\mathcal {Y}$ are closed under extensions in $\mathcal{A}$ and $\mathcal{B}$ respectively. Let $0\rightarrow\left(\begin{smallmatrix}X'\\Y' \end{smallmatrix}\right)_{\varphi'}\rightarrow \left(\begin{smallmatrix}X\\Y\end{smallmatrix}\right)_{\varphi}\rightarrow\left(\begin{smallmatrix}X''\\Y''\end{smallmatrix}\right)_{\varphi''}\rightarrow0 $ be an exact sequence in $(\mathcal{B}\downarrow T)$ with $\left(\begin{smallmatrix}X'\\Y' \end{smallmatrix}\right)_{\varphi'}\in\langle{\bf h}(\mathcal{X},~\mathcal{Y})\rangle$. Note that $T$ is $\mathcal {X}$-exact, so we have the following commutative diagram of exact sequences:
$$\xymatrix{0\ar[r]&Y'\ar[r]\ar[d]^{\varphi'}&Y\ar[r]\ar[d]^\varphi&Y''\ar[r]\ar[d]^{\varphi''}&0\\
    0\ar[r]&TX' \ar[r]&TX\ar[r]&TX'' \ar[r]&0.}$$
Note that $\varphi'$ is an epimorphism and $\mbox{ker}\varphi'\in \mathcal{Y}$ by \prpref{hxyD}. Thus we have that $\mbox{coker}\varphi\cong{\mbox{coker}}\varphi''$ and
$0\rightarrow\mbox{ker}\varphi'\rightarrow\mbox{ker}\varphi\rightarrow\mbox{ker}\varphi''\rightarrow0$ is exact in $\mathcal{B}$. Therefore, we get that $\varphi$ is surjective if and only if $\varphi''$ is surjective. Since $\mathcal{Y}$ is coresolving in $\mathcal{B}$ by hypothesis, $\mbox{ker}\varphi\in \mathcal {Y}$ if and only if $\mbox{ker}\varphi''\in \mathcal {Y}$. Note that $0\rightarrow X' \rightarrow X\rightarrow X''\rightarrow0$ is an exact sequence in $\mathcal{A}$ and $\mathcal{X}$ is coresolving in $\mathcal{A}$. It follows that $X\in\mathcal{X}$ if and only if $X''\in\mathcal{X}$. So  $\left(\begin{smallmatrix}X\\Y \end{smallmatrix}\right)_{\varphi}\in\langle{\bf h}(\mathcal{X},~\mathcal{Y})\rangle$ if and only if $\left(\begin{smallmatrix}X''\\Y'' \end{smallmatrix}\right)_{\varphi'}\in\langle{\bf h}(\mathcal{X},~\mathcal{Y})\rangle$ by \prpref{hxyD} .
\end{proof}

\begin{thm}
Let $\mathcal{A}$ and $\mathcal{B}$ both have enough projective objects and enough injective objects.
If $T:\mathcal{A}\rightarrow\mathcal{B}$ is $\mathcal{X}$-exact, then $(^{\bot}\mathcal{X},~\mathcal{X})$ is a (hereditary) cotorsion pair in $\mathcal{A}$ and $(^{\bot}\mathcal{Y},~\mathcal{Y})$) is a (hereditary) cotorsion pair in $\mathcal{B}$ if and only if $((\begin{smallmatrix}^{\bot}\mathcal{X}\\^{\bot}\mathcal{Y}\end{smallmatrix}),\langle{\bf h}(\mathcal{X},~\mathcal{Y})\rangle)$ is a (hereditary) cotorsion pair in $(\mathcal{B}\downarrow T)$ and $\mathcal{X}$ and $\mathcal{Y}$ are closed under extensions.
\end{thm}

\begin{proof}
``$\Rightarrow$". Let $T:\mathcal{A}\rightarrow\mathcal{B}$ be $\mathcal{X}$-exact, $(^{\bot}\mathcal{X},~\mathcal {X})$ be a cotorsion pair in $\mathcal{A}$  and $(^{\bot}\mathcal{Y},~\mathcal {Y})$ a cotorsion pair in $\mathcal{B}$. Then all projective objects belong to $^{\bot}\mathcal{Y}$. By \prpref{rightoth}, one has
$$\langle{\bf h}(\mathcal{X},~\mathcal{Y})\rangle=\langle{\bf h}((^{\bot}\mathcal{X}){^{\bot}},~{(^{\bot}\mathcal{Y})^{\bot})}\rangle={(\begin{smallmatrix}^{\bot}\mathcal{X}\\
^{\bot\mathcal{Y}}\end{smallmatrix})^{\bot}}\cap{(\begin{smallmatrix}0\\ \mathcal{P} \end{smallmatrix})^{\bot}}={(\begin{smallmatrix}^{\bot}\mathcal{X}\\
^{\bot\mathcal{Y}}\end{smallmatrix})^{\bot}}.$$
\noindent It follows from \prpref{Xexact} that $^{\bot}{\langle{\bf h}(\mathcal{X},~\mathcal {Y})\rangle}={(\begin{smallmatrix}^{\bot}\mathcal{X}\\^{\bot}\mathcal {Y}\end{smallmatrix})}$, so $((\begin{smallmatrix}^{\bot}\mathcal{X}\\^{\bot}\mathcal{Y}\end{smallmatrix}),~\langle{\bf h}(\mathcal{X},~\mathcal{Y})\rangle)$ is a cotorsion pair in $(\mathcal{B}\downarrow T)$.

``$\Leftarrow$". Let $((\begin{smallmatrix}^{\bot}\mathcal{X}\\^{\bot}\mathcal{Y}\end{smallmatrix}),
~\langle{\bf h}(\mathcal{X},~\mathcal{Y})\rangle)$ be a cotorsion pair in $(\mathcal{B}\downarrow T)$ and $\mathcal{X}$, $\mathcal {Y}$ be closed under extensions. It is sufficient to show that  ${(^{\bot}\mathcal {X})^{\bot}}\subseteq\mathcal {X}$ and ${(^{\bot}\mathcal {Y})^{\bot}}\subseteq\mathcal{Y}$. Assume that $X\in{(^{\bot}\mathcal{X})^{\bot}}$ and $Y\in{(^{\bot}\mathcal{Y})^{\bot}}$. Then for any $M\in\,^{\bot}\mathcal{X}$ and $N\in\,^{\bot}\mathcal{Y}$, it is clear that the following exact sequences
$$0\rightarrow\left(\hspace{-0.15cm}\begin{array}{c} X\\TX\\\end{array}\hspace{-0.15cm}\right)\rightarrow \left(\hspace{-0.15cm}\begin{array}{c} A\\B\\\end{array}\hspace{-0.15cm}\right)\rightarrow
\left(\hspace{-0.15cm}\begin{array}{c} M\\N\\\end{array}\hspace{-0.15cm}\right)\rightarrow0~\rm{and}~
0\rightarrow \left(\hspace{-0.15cm}\begin{array}{c} 0\\Y\\\end{array}\hspace{-0.15cm}\right)\rightarrow \left(\hspace{-0.15cm}\begin{array}{c} M\\C\\\end{array}\hspace{-0.15cm}\right)\rightarrow \left(\hspace{-0.15cm}\begin{array}{c} M\\N\\\end{array}\hspace{-0.15cm}\right)\rightarrow0 $$
are split. This implies that $\left(\begin{smallmatrix}X\\TX \end{smallmatrix}\right), \left(\begin{smallmatrix}0\\Y\end{smallmatrix}\right)\in {(\begin{smallmatrix}^{\bot}\mathcal{X}\\^{\bot}\mathcal {Y}\end{smallmatrix})^{\bot}}=\langle{\bf h}(\mathcal{X},~\mathcal{Y})\rangle$. Therefore $X\in\mathcal{X}, Y\in\mathcal{Y}$ by \prpref{hxyD}.
\end{proof}

\begin{prp}\label{prp:cotorsionpair}
Let $T:\mathcal{A}\rightarrow\mathcal{B}$ be $\mathcal{X}$-exact and $\mathcal{A}$ and $\mathcal{B}$ both have enough projective objects and enough injective objects. If $(^{\bot}\mathcal{X},~\mathcal{X})$ and $(^{\bot}\mathcal{Y},~\mathcal{Y})$ are complete cotorsion pairs in $\mathcal{A}$ and $\mathcal{B}$, respectively, then so is $(\left(\begin{smallmatrix}^{\bot}\mathcal{X}\\^{\bot}\mathcal{Y}\end{smallmatrix}\right),~\langle{\bf h}(\mathcal{X},~\mathcal{Y})\rangle)$. Moreover, the converse holds if $T(^{\bot}\mathcal{X}\cap\mathcal {X})\subseteq{^{\bot}\mathcal{Y}}$ and $\mathcal{X}$, $\mathcal{Y}$ are closed under extensions.
\end{prp}
\begin{proof}Assume that $(^{\bot}\mathcal{X},~\mathcal{X})$ is a complete cotorsion pair in $\mathcal{A}$ and $(^{\bot}\mathcal{Y},~\mathcal{Y})$ is a complete cotorsion pair in $\mathcal{B}$. For any $\left(\begin{smallmatrix}A\\B\end{smallmatrix}\right)\in(\mathcal{B}\downarrow T)$, there exists an exact sequence $0\rightarrow X\rightarrow U\rightarrow A\rightarrow0$ in $\mathcal{A}$ with $U\in^{\bot}\mathcal{X}$ and $X\in \mathcal{X}$. Since $T$ is $\mathcal{X}$-exact, there is a pullback diagram
$$\xymatrix{0\ar[r]&TX\ar[r]\ar@{=}[d]&C\ar[r]\ar[d]&B\ar[r]\ar[d]&0\\
0\ar[r]&TX \ar[r]&TU\ar[r]&TA \ar[r]&0.\\}$$

Furthermore, we have an exact sequence $0\rightarrow D\rightarrow V\rightarrow C\rightarrow0$ in $\mathcal{A}$ with $V\in^{\bot}\mathcal{Y}$ and $D\in\mathcal{Y}$. Consider the following pullback diagram
$$\xymatrix{&0\ar[d]&0\ar[d]&\\&D\ar@{=}[r]\ar[d]&D\ar[d]&\\
0\ar[r]&Y\ar[r]\ar[d]&V\ar[r]\ar[d]&B\ar[r]\ar@{=}[d]&0\\
0\ar[r]&TX\ar[r]\ar[d]&C\ar[d]\ar[r]&B \ar[r]&0\\
&0&0&\\}$$

Thus we get an exact sequence $0\rightarrow\left(\begin{smallmatrix}X\\Y\end{smallmatrix}\right){\rightarrow} \left(\begin{smallmatrix}U\\V\end{smallmatrix}\right)\rightarrow\left(\begin{smallmatrix}A\\B\end{smallmatrix}\right)\rightarrow0$ in $(\mathcal{B}\downarrow T)$ with $\left(\begin{smallmatrix}U\\V\end{smallmatrix}\right)\in(\begin{smallmatrix}^{\bot}\mathcal{X}\\^{\bot}\mathcal {Y}\end{smallmatrix})$ and $\left(\begin{smallmatrix}X\\Y\end{smallmatrix}\right)\in\langle{\bf h}(\mathcal {X},~\mathcal {Y})\rangle$. Note that $\mathcal{A}$ and $\mathcal{B}$ have enough projective objects, so is $(\mathcal{B}\downarrow T)$. It follows that $((\begin{smallmatrix}^{\bot}\mathcal{X}\\^{\bot}\mathcal {Y}\end{smallmatrix}),~\langle{\bf h}(\mathcal{X},~\mathcal{Y})\rangle)$ is a complete cotorsion pair in $(\mathcal{B}\downarrow T)$.

Conversely, assume that $((\begin{smallmatrix}^{\bot}\mathcal{X}\\^{\bot}\mathcal {Y}\end{smallmatrix}),~\langle{\bf h}(\mathcal{X},~\mathcal{Y})\rangle)$ is a complete cotorsion pair in $(\mathcal{B}\downarrow T)$, $T(^{\bot}\mathcal{X} \cap \mathcal{X})\subseteq{^{\bot}\mathcal{Y}}$ and $\mathcal{X}, \mathcal{Y}$ are closed under extensions. Then for any $A\in\mathcal{A}$, we have an exact sequence  $0\rightarrow\left(\begin{smallmatrix}A\\TA\end{smallmatrix}\right){\rightarrow} \left(\begin{smallmatrix}X\\Y\end{smallmatrix}\right)\rightarrow\left(\begin{smallmatrix}U\\V\end{smallmatrix}\right)\rightarrow0$ in $(\mathcal{B}\downarrow T)$ with $\left(\begin{smallmatrix}X\\Y\end{smallmatrix}\right)\in\langle{\bf h}(\mathcal {X},~\mathcal {Y})\rangle$ and $(\begin{smallmatrix}U\\V\end{smallmatrix})\in(\begin{smallmatrix}^{\bot}\mathcal{X}\\^{\bot}\mathcal {Y}\end{smallmatrix})$. Thus we have an exact sequence $0\rightarrow A\rightarrow X\rightarrow U\rightarrow0$ in ${\mathcal{A}}$ with $X\in\mathcal{X}$ and $U\in^{\bot}\mathcal{X}$ which means that $(^{\bot}\mathcal{X},~\mathcal {X})$ is a complete cotorsion pair in $\mathcal{A}$.
Furthermore, for any $B\in\mathcal{B}$, we have an exact sequence  $0\rightarrow\left(\begin{smallmatrix}0\\B\end{smallmatrix}\right){\rightarrow} \left(\begin{smallmatrix}M\\C\end{smallmatrix}\right)_\varphi{\rightarrow}\left(\begin{smallmatrix}M\\N \end{smallmatrix}\right)_{\phi}\rightarrow0$ in $(\mathcal{B}\downarrow T)$  with $\left(\begin{smallmatrix}M\\C\end{smallmatrix}\right)_\varphi\in\langle{\bf h}(\mathcal {X},~\mathcal {Y})\rangle$ and $\left(\begin{smallmatrix}M\\N\end{smallmatrix}\right)_\phi\in\left(\begin{smallmatrix}^{\bot}\mathcal{X}\\^{\bot}\mathcal {Y}\end{smallmatrix}\right)$. Thus we obtain that $M\in^{\bot}\mathcal{X}\cap\mathcal{X}$, ker$\varphi\in\mathcal{Y}$, $N\in^{\bot}\mathcal{Y}$ and $\varphi$ is a epimorphism. Consequently,
we have the following commutative diagram of exact rows and columns:
$$\xymatrix{&&0\ar[d]&0\ar[d]\\
0\ar[r]&B\ar[r]\ar@{=}[d]&\rm{ker}\varphi\ar[r]\ar[d]&\rm{ker}\phi\ar[r]\ar[d]&0\\
0\ar[r]&B\ar[r]&C\ar[d]^{\varphi}\ar[r]&N\ar[r]\ar[d]^\phi&0\\
&&TM\ar@{=}[r]\ar[d]&TM\ar[d]\\
&&0&0.\\}$$
By hypotheses, $TK\in^{\bot}\mathcal{Y}$. It follows that the middle column splits which induces that the right column is split. Thus ker$\phi$ is a direct summand of $Y$, i.e., ker$\phi\in^{\bot}\mathcal{Y}$. It follows from the exact sequence $0\rightarrow B\rightarrow\rm{ker}\varphi\rightarrow\rm{ker}\phi\rightarrow0$ that $(^{\bot}\mathcal{Y},~\mathcal{X})$ is a complete cotorsion pair in $\mathcal{B}$.
\end{proof}

\begin{rmk} {\rm If we replace the condition ``$T(\mathcal {Y} \cap \mathcal {Y}^{\bot})\subseteq{\mathcal {X}^{\bot}}$" with ``$T(\mathcal {Y} \cap \mathcal {Y}^{\bot})\subseteq{\mathcal {X}}$" in \prpref{cotorsionpair}, the result still holds. In the case $T(\mathcal {Y} \cap \mathcal {Y}^{\bot})\subseteq{\mathcal {X}}$, for any $A\in\mathcal{A}$, it is easy to check that the exact sequence $0\to K\to M\to A\to 0$ in the third commutative diagram in \prpref{cotorsionpair} satisfies that $K\in{\mathcal{X}^{\perp}}$ and $M\in{\mathcal{X}}$. This implies that $(\mathcal {X}, \mathcal {X}^{\bot})$ is a complete cotorsion pair in $\mathcal{A}$.}
\end{rmk}

Let $\mathcal{L}$ be a class of objects in an abelian category $\mathcal{D}$. We denote by $\mathsf{Smd}(\mathcal{L})$ the class of direct summands of objects in $\mathcal{L}$.

\begin{lem}\label{lem:cotorsionpair}
Let $\mathcal{D}$ be an abelian category with enough projective objects. If $\mathcal L$ is special preeveloping in $\mathcal{D}$, then $(^\bot\mathsf{Smd}(\mathcal{L}),~\mathsf{Smd}(\mathcal{L}))$ is a complete cotorsion pair.
\end{lem}
\begin{proof}
Clearly, $^\bot\mathsf{Smd}(\mathcal{L})=\,^\bot\mathcal{L}$. Then it is sufficient to show that $\mathsf{Smd}(\mathcal{L})$ is a special preenveloping class. We claim that ${(^{\bot}\mathsf{Smd}(\mathcal{L}))^{\bot}}\subseteq\mathsf{Smd}(\mathcal{L})$. For any $X\in{(^{\bot}\mathsf{Smd}(\mathcal{L}))^{\bot}}$, there is an exact sequence $\xi:0\rightarrow\overline{L}\rightarrow L \rightarrow X\rightarrow0$ with $\overline{L}\in ^{\bot}\mathcal{L}$ and $L\in\mathcal{L}$ since $\mathcal{L}$ is  special preenveloping. Thus $\xi$ is split as $^{\bot}\mathsf{Smd}(\mathcal{L})=^{\bot}\mathcal{L}$. It follows that $X\in\mathsf{Smd}(\mathcal{L})$. Therefore, ${(^{\bot}\mathsf{Smd}(\mathcal{L}))^{\bot}}=\mathsf{Smd}(\mathcal{L})$.
\end{proof}

\section{Gorenstein injective objects}

\noindent In this section, we characterize when special preenveloping classes in abelian categories $\mathcal{A}$ and $\mathcal{B}$ can induce special preenveloping classes in $(\mathcal{B}\hspace{-0.05cm}\downarrow\hspace{-0.05cm}T)$.

\begin{dfn}\label{dfn:cocompatible}
The left exact functor $T:\mathcal{A}\rightarrow \mathcal{B}$ is $\emph{cocompatible}$, if the following two conditions hold:
\begin{enumerate}
\item[(C1)] $TE^\bullet$ is exact for each exact sequence $E^\bullet$ of injective objects in $\mathcal{A}$.
\item[(C2)] ${\rm Hom}_{\mathcal{B}}(TI,~J^\bullet)$ is  exact for each totally acyclic complex of injective objects $J^\bullet$ in $\mathcal{B}$ and each injective object $I$ in $\mathcal{A}$.
\end{enumerate}
Moreover, $T:\mathcal{A}\rightarrow \mathcal{B}$ is called $\emph{weak~cocompatible}$, if $T$ satisfies conditions (W1) and (C2), where
\begin{enumerate}
\item[(W1)] $TE^\bullet$ is exact for each totally acyclic complex of injective objects $E^\bullet$ in $\mathcal{A}$.
\end{enumerate}
\end{dfn}

\begin{prp}\label{00}
Let $M$ be an $R$-$S$-bimodule and $T={\rm Hom}_R(M,~-)$. Then
\begin{enumerate}
\item If {\rm pd}$_RM$ is finite, then $T$ satisfies (C1).
\item If {\rm fd}$M_S$ is finite, then $T$ satisfies (C2).
\end{enumerate}
\end{prp}
\begin{proof}
(1) Let $E^\bullet=\cdots\rightarrow E_{1}\stackrel{d_{1}}\rightarrow E_0\stackrel{d_0}\rightarrow Q_{-1}\rightarrow\cdots$ be an exact sequence of injective left $R$-modules and pd$_RM=n$. By Dimension Shifting, we have ${\rm Ext}_R^1(M,~{\ker} d_i)={\rm Ext}_R^{n+1}(M,~{\ker} d_{i-n})=0$ for any integer $i$. It follows that $TE^\bullet$ is still exact.

(2) Let {\rm fd}$M_S$ is finite. For any injective left $R$-module $I$, {\rm id$_R$Hom}$(M,~I)$ is finite, since {\rm fd}$M_S$ is finite and $I$ is injective. Then ${\rm Ext}_R^1({\rm Hom}(M,~I),~{\ker} \phi_i)=0$ for any totally acyclic complex of injective $R$-modules $J^\bullet=\cdots\rightarrow J_{1}\stackrel{\phi_{1}}\rightarrow J_0\stackrel{\phi_0}\rightarrow J_{-1}\rightarrow\cdots$.
\end{proof}

\begin{lem}\label{lem:Gorensteiniinj}
Let  $G$ be an object in $\mathcal{B}$ and $L$ an object in $\mathcal{A}$. The following statements hold:
\begin{enumerate}
\item If ${\bf h}(0,G)$ is a Gorenstein injective object, then $G$ is Gorenstein injective.
\item If ${\bf h}(L,0)$ is a Gorenstein injective object, then $L$ is Gorenstein injective.
\end{enumerate}
\end{lem}
\begin{proof}
(1) Let ${\bf h}(0,G)$ be a Gorenstein injective object. There is a totally acyclic complex
$${{\bf h}(I^\bullet,~E^\bullet)}=\cdots\rightarrow{\bf h}(I_{1},~E_{1}){\rightarrow} {\bf h}(I_{0},~E_{0})\rightarrow
{\bf h}(I_{-1},~E_{-1})\rightarrow\cdots$$
of injective objects in $(\mathcal{B}\downarrow T)$ with $Z_0({{\bf h}(I^\bullet,~E^\bullet)})={\bf h}(0,~G)$. By definition ${{\bf h}(I^\bullet,~E^\bullet)}$ is ${\rm Hom}_{(\mathcal{B}\downarrow T)}({{\bf h}(0,~X)},~-)$-exact for each injective object $X$ in $\mathcal{B}$. Now one has
$$\begin{array}{lll}{\rm Hom}_{(\mathcal{B}\downarrow T)}({\bf h}(0,~X),~{\bf h}(I^\bullet,~E^\bullet))&\simeq{\rm Hom}_{\mathcal{A}\times\mathcal{B}}({\bf q}({\bf h}(0,~X)),~(I^\bullet,~E^\bullet))&\\&={\rm Hom}_{\mathcal{B}}(X,~E^\bullet)&\end{array}$$
which implies that $E^\bullet=\cdots\rightarrow E_{1}\rightarrow E_0\rightarrow E_{-1}\rightarrow\cdots$ is ${\rm Hom}_{\mathcal{B}}(X,-)$-exact. Clearly, $G=Z_0(E^\bullet)$ and $G$ is Gorenstein injective.

(2) Let ${\bf h}(L,~0)$ be a Gorenstein injective object. Then there is a totally acyclic complex
$${{\bf h}(J^\bullet,~F^\bullet)}=\cdots\rightarrow{\bf h}(J_{1},~F_{1}){\rightarrow} {\bf h}(J_{0},~F_{0})\rightarrow
{\bf h}(J_{-1},~F_{-1})\rightarrow\cdots$$
of injective objects in $(\mathcal{B}\downarrow T)$ with $Z_0({{\bf h}(J^\bullet,~F^\bullet)})={\bf h}(L,~0)$. By definition ${{\bf h}(J^\bullet,F^\bullet)}$ is ${\rm Hom}_{(\mathcal{B}\downarrow T)}({\bf h}(Y,~0),~-)$-exact for each injective object $Y$ in $\mathcal{A}$. Then we have $$\begin{array}{lll}{\rm Hom}_{(\mathcal{B}\downarrow T)}({\bf h}(Y,~0),~{\bf h}(J^\bullet,~F^\bullet))&\simeq{\rm Hom}_{\mathcal{A}\times\mathcal{B}}({\bf q}({\bf h}(Y,~0)),(J^\bullet,~F^\bullet))&\\&\simeq {\rm Hom}_\mathcal{A}(Y,~J^\bullet)\times{\rm Hom}_\mathcal{B}(TY,~F^\bullet)&\end{array}$$
which implies that $J^\bullet=\cdots\rightarrow J_{1}\rightarrow J_0\rightarrow J_{-1}\rightarrow\cdots$ is ${\rm Hom}_{\mathcal{A}}(Y,-)$-exact. Clearly, $L=Z_0({J}^\bullet)$ and $L$ is Gorenstein injective.
\end{proof}

\begin{lem}\label{lem:W1C2}
For the comma category  $(\mathcal{B}\downarrow T)$, one has
\begin{enumerate}
\item $T$ satisfies (W1) if and only if ${\bf h}(G,0)$ is a Gorenstein injective object for each Gorenstein injective object $G$ in $\mathcal{A}$.
\item $T$ satisfies (C2) if and only if ${\bf h}(0,L)$ is a Gorenstein injective object for each Gorenstein injective object $L$ in $\mathcal{B}$.
\end{enumerate}
\end{lem}
\begin{proof}
(1)$``\Leftarrow"$. Assume that $E^\bullet=\cdots\rightarrow E_{1}\stackrel{d_1}\rightarrow E_0\stackrel{d_0}\rightarrow E_{-1}\rightarrow\cdots$ is a totally acyclic complex of injective objects in $\mathcal{A}$ and  $G_0=Z_0(E^\bullet)$. There is an exact sequence $0\rightarrow G_1\rightarrow E_1\rightarrow G_0\rightarrow0$ with $E_0$ injective and $G_1$ Gorenstein injective. Consequently, ${\bf h}(G_0,0)$ is Gorenstein injective in $(\mathcal{B}\downarrow T)$ by hypothesis. Thus there is an exact sequence $0\rightarrow\left(\begin{smallmatrix}G_A\\G_B\end{smallmatrix}\right){\rightarrow} \left(\begin{smallmatrix}E_A\\E_B\end{smallmatrix}\right)\rightarrow\left(\begin{smallmatrix}G_0\\TG_0\end{smallmatrix}\right)\rightarrow0$ in $(\mathcal{B}\downarrow T)$ with $\left(\begin{smallmatrix}E_A\\E_B\end{smallmatrix}\right)$ injective and $\left(\begin{smallmatrix}G_A\\G_B\end{smallmatrix}\right)$ Gorenstein injective. Since $E_A$ is injective, we have the following commutative diagram with exact rows:
$$\xymatrix{
0\ar[r]&G_A\ar[r]\ar[d]&E_A\ar[r]\ar[d]&G_0\ar[r]\ar@{=}[d]&0\\
    0\ar[r]&G_1\ar[r]&E_1\ar[r]&G_0\ar[r]&0.\\
    }$$
Therefore, we obtain the following exact commutative diagram
$$\xymatrix{
0\ar[r]&G_B\ar[r]\ar[d]&E_B\ar[r]\ar[d]&TG_0\ar[r]\ar@{=}[d]&0\\
0\ar[r]&TG_A\ar[r]\ar@{=}[d]&TE_A\ar[r]\ar[d]&TG_0\ar@{=}[d]&\\
0\ar[r]&TG_1\ar[r]&TE_1\ar^{Tf}[r]&TG_0&\\
    }$$
which implies that $Tf$ is surjective, as required.

$``\Rightarrow"$. Let $G$ be a Gorenstein injective object in $\mathcal{A}$. Then there is a totally acyclic complex of injective objects $I^\bullet=\cdots\rightarrow I_{1}\stackrel{d_1}\rightarrow I_0\stackrel{d_0}\rightarrow I_{-1}\rightarrow\cdots$ in $\mathcal{A}$ such that $G={\ker} d_0$. By hypothesis,
$${\bf h}(I^\bullet,0)=\cdots\longrightarrow{\bf h}(I_{1},0)\xlongrightarrow {{\bf h}(d_{1},0)} {\bf h}(I_{0},0)\xlongrightarrow {{\bf h}(d_{0},0)}  {\bf h}(I_{-1},0)\longrightarrow\cdots$$
is an exact sequence of injective objects in$(\mathcal{B}\downarrow T)$. Then we have
$${\rm Hom}_{(\mathcal{B}\downarrow T)}({\bf h}(J,F),{\bf h}(I^\bullet,0))\simeq{\rm Hom}_{\mathcal{A}\times\mathcal{B}}({\bf q}({\bf h}(J,F)),(I^\bullet,0))={\rm Hom}_{\mathcal{A}}(J,I^\bullet)$$
for each injective object ${\bf h}(J,F)$ in $(\mathcal{B}\downarrow T)$. Thus, ${\bf h}(I^\bullet,0)$ is totally acyclic. Then ${\bf h}(G,0)$ is a Gorenstein injective object by ${\bf h}(G,0)={\ker}\,\textbf{h}(d_{0},0)$.

(2) Let $F^\bullet=\cdots\rightarrow F_{1}\stackrel{\tau_{1}}\rightarrow F_0\stackrel{\tau_0}\rightarrow F_{-1}\rightarrow\cdots$ be a totally acyclic complex of injective objects in $\mathcal{B}$. Then there exists an exact sequence
$${\bf h}(0,F^\bullet)=\cdots\longrightarrow{\bf h}(0,F_{1}) \xlongrightarrow {{\bf h}(0,\tau_{1})} {\bf h}(0,F_{0}) \xlongrightarrow {{\bf h}(0,\tau_{0})}
{\bf h}(0,F_{-1}) \longrightarrow\cdots$$
in $(\mathcal{B}\downarrow T)$ with each term injective. Then one has
\begin{align*}
{\rm Hom}_{(\mathcal{B}\downarrow T)}({\bf h}(0,~F^\bullet),~{\bf h}(X,~Y))&\simeq{\rm Hom}_{\mathcal{A}\times\mathcal{B}}({\bf q}({\bf h}(X,~Y)),~(0,~F^\bullet)) \\&={\rm Hom}_{\mathcal{B}}(Y\oplus TX,~F^{\bullet})
\end{align*}
for each injective object ${\bf h}(X,~Y)$ in $(\mathcal{B}\downarrow T)$.
It follows that ${\rm Hom}_{\mathcal{B}}(TX,~F^{\bullet})$ is exact for each injective object $X$ in $\mathcal{A}$ if and only if ${\bf h}(0,~F^\bullet)$ is a totally acyclic complex of injective objects in  $(\mathcal{B}\downarrow T)$ for each totally acyclic complex of injective objects $F^\bullet$ in $\mathcal{B}$.
\end{proof}

\begin{prp}\label{prp:compatible}
Let $(\mathcal{B}\downarrow T)$ be a comma category. Then $\langle{\bf h}(\mathcal{GI}_\mathcal{A},~\mathcal {GI}_\mathcal{B})\rangle\subseteq\mathcal{GI}_{(\mathcal{B}\downarrow T)}$ if and only if $T$ is weak cocompatible. Moreover, if $T$ is cocompatible, then
$\mathcal {GI}_{(\mathcal{B}\downarrow T)}=\langle{\bf h}(\mathcal {GI}_\mathcal{A},~\mathcal {GI}_\mathcal{B})\rangle$.
\end{prp}
\begin{proof} By Lemmas \ref{lem:Gorensteiniinj} and \ref{lem:W1C2}, $\langle{\bf h}(\mathcal {GI}_\mathcal{A},~\mathcal {GI}_\mathcal{B})\rangle\subseteq\mathcal{GI}_{(\mathcal{B}\downarrow T)}$ is satisfied. It is sufficient to show $\mathcal {GI}_{(\mathcal{B}\downarrow T)}\subseteq\langle{\bf h}(\mathcal {GI}_\mathcal{A},~\mathcal {GI}_\mathcal{B})\rangle$ provided that $T$ is cocompatible.

 Assume that $\left(\begin{smallmatrix}H\\G\end{smallmatrix}\right)_\varphi\in\mathcal {GI}_{(\mathcal{B}\downarrow T)}$ and there is a totally acyclic complex
$${\bf h}(I^\bullet,~E^\bullet)=\cdots\longrightarrow{\bf h}(I_{1},~E_{1})\xlongrightarrow {{\bf h}(d_{1},~\delta_{1})} {\bf h}(I_{0},~E_0)\xlongrightarrow {{\bf h}(d_{0},~\delta_0)}
{\bf h}(P_{-1},~E_{-1})\longrightarrow\cdots$$ of injective objects
in $(\mathcal{B}\downarrow T)$ with $\left(\begin{smallmatrix}H\\G\end{smallmatrix}\right)_\varphi = {\ker}~{\bf h}(d_{0},\delta_0)$. Then $I^\bullet$ is an exact sequence  of injective objects in $\mathcal{A}$, and $TI^\bullet$ is exact since $T$ is cocompatible. Thus we have the following exact commutative diagram:
$$\xymatrix{
0\ar[r]&G_i\ar[r]\ar[d]^{\varphi_i}&E_i\oplus TI_i\ar[r]\ar[d]^{{(0,1)}}&G_{i-1}\ar[d]^{^{\varphi_{i-1}}}\ar[r]&0\\
0\ar[r]&TH_i\ar[r]&TI_{i}\ar[r]&TH_{i-1}\ar[r]&0,\\}$$
where $G_i=\rm{ker}{\delta}_{i}, {\left(\begin{smallmatrix}{H_i}\\{G_i} \end{smallmatrix}\right)_{\varphi_{i}}} = {\rm{ker}}~{\bf h}(d_{i},~\delta_{i})$ and $\varphi_i$ is canonical induced. In particular, $G_0=G, \ H_0=H$ and $\varphi_0 = \varphi$. It follows that each $\varphi_i$ is surjective and there is an exact sequence $0\rightarrow\rm{ker} \varphi_i\rightarrow \emph{E}_\emph{i} \rightarrow\rm{ker}\varphi_{i-1}\rightarrow0.$
Therefore, we have an exact sequence $E^\bullet=\cdots\rightarrow E_{1}\stackrel{\tau_{1}}\rightarrow E_0\stackrel{\tau_{0}}\rightarrow P_{-1}\rightarrow\cdots$ with each term injective in $\mathcal{B}$ and ${\ker} \tau_0={\ker} \varphi$. For each injective object $J$ in $\mathcal{B}$, one has
$${{\rm Hom}_{(\mathcal{B}\downarrow T)}({\bf h}(0,~J),{\bf h}(I^\bullet,~E^\bullet))}\simeq{\rm Hom}_{\mathcal{A}\times\mathcal{B}}({\bf q}({\bf h}(0,~J),(I^\bullet,~E^\bullet)))={\rm Hom}_{\mathcal{B}}(J,~E^\bullet).$$
By definition ${{\rm Hom}_{(\mathcal{B}\downarrow T)}}({\bf h}(I^{\bullet},~E^{\bullet}),{\bf h}(0,~J))$ is exact. It follows that $E^\bullet$ is totally acyclic. This implies that $\ker\varphi$ is a Gorenstein injective object in $\mathcal{B}$.

Let $C$ be an injective object in $\mathcal{A}$. Applying ${\rm Hom}_{(\mathcal{B}\downarrow T)}({\bf h}(C,~0),~-)$ to the totally acyclic complex of injective objects ${\bf h}(I^\bullet,~E^\bullet)$ in $(\mathcal{B}\downarrow T)$, we obtain that $I^\bullet=\cdots\rightarrow I_{1}\rightarrow I_0\rightarrow I_{-1}\rightarrow\cdots$ is ${\rm Hom}_{\mathcal{B}}(C,~-)$-exact and the proof is similar to the above. So $H=Z_0({I}^\bullet)$ is Gorenstein injective.
\end{proof}

\begin{thm}\label{thm:special preenveloping}
Let $\mathcal{A}$ and $\mathcal{B}$ both have enough projective objects and enough injective objects.
\begin{enumerate}
\item Assume that $\mathcal{X}$  is a subclass of $\mathcal{A}$ with $0\in{\mathcal{X}}$,  $\mathcal{Y}$ is a subclass of $\mathcal{B}$ with $0\in{\mathcal{Y}}$ and $T:\mathcal{A}\rightarrow\mathcal{B}$ is a $\mathcal {X}$-exact functor. If $\mathcal {X}$ and  $\mathcal {Y}$ are special prenveloping, then $\langle{\bf h}(\mathcal {X},~\mathcal {Y})\rangle$ is also special prenveloping in $(\mathcal{B}\downarrow T)$. Moreover, the converse holds when $T(^{\bot}\mathcal{X}\bigcap\mathcal{X})\subseteq{^{\bot}\mathcal{Y}}$ and $\mathcal X, \mathcal Y$ are closed under extensions.
\label{11}\item If $T:\mathcal{A}\rightarrow\mathcal{B}$ is a cocompatible functor, then $\mathcal{GI}_\mathcal{A}$ and $\mathcal{GI}_\mathcal{B}$ are special preenveloping in $\mathcal{A}$ and $\mathcal{B}$, respectively if and only if $\mathcal{GI}_{(\mathcal{B}\downarrow T)}$ is special preenveloping in $(\mathcal{B}\downarrow T)$.
\end{enumerate}
\end{thm}
\begin{proof}
At first, we claim that $T$ is $\mathsf{Smd}(\mathcal {X})$-exact. For each exact sequence $0\rightarrow X\rightarrow A_1\rightarrow A_0\rightarrow0$ in $\mathcal{A}$ with ${X}\in\mathsf{Smd}(\mathcal{X})$, there exists $M$ such that $X\oplus M\in \mathcal{X}$. Then there is an induced exact sequence
$0\rightarrow X\oplus M\rightarrow A_1\oplus M\rightarrow A_0\rightarrow0$. Consequently, we obtain the following commutative diagram
$$\xymatrix{
0\ar[r]&TX\oplus TM\ar[r]\ar[d]&TA_1\oplus TM\ar[r]\ar[d]&TA_0\ar[r]\ar@{=}[d]&0\\
    0\ar[r]&TX\ar[r]& TA_1\ar[r]&TA_0 \ar[r]&0,\\
    }$$
where the top row is exact since $T$ is $\mathcal{X}$-exact. It follows that the bottom row is exact.

Since $(^\bot\mathsf{Smd}(\mathcal{X}),~\mathsf{Smd}(\mathcal{X}))$ and $(^\bot\mathsf{Smd}(\mathcal {Y}),~ \mathsf{Smd}(\mathcal {Y}))$ are complete cotorsion pairs by \lemref{cotorsionpair}, the pair $((\begin{smallmatrix}^{\bot}\mathsf{Smd}(\mathcal{X})\\^{\bot}\mathsf{Smd}(\mathcal{Y})\end{smallmatrix}), ~\langle{\bf p}(\mathsf{Smd}(\mathcal{X}),\mathsf{Smd}(\mathcal {Y}))\rangle)$ is a complete cotorsion pair in $(\mathcal{B}\downarrow T)$ by \prpref{cotorsionpair}. Thus for each object $\left(\begin{smallmatrix}A\\B\end{smallmatrix}\right)\in(\mathcal{B}\downarrow T)$, there is an exact sequence $0\rightarrow\left(\begin{smallmatrix}A\\B\end{smallmatrix}\right){\rightarrow}
(\begin{smallmatrix}\overline{M}\\ \overline{Y}\end{smallmatrix})_{\overline{\varphi}}\rightarrow
(\begin{smallmatrix}\overline{C}\\ \overline{D}\end{smallmatrix})_{\overline{\psi}}\rightarrow0$ with $(\begin{smallmatrix}\overline{M}\\ \overline{Y}\end{smallmatrix})_{\overline{\varphi}}\in\langle{\bf h}(\mathsf{Smd}(\mathcal {X}),~\mathsf{Smd}(\mathcal {Y}))\rangle$ and $(\begin{smallmatrix}\overline{C}\\ \overline{D}\end{smallmatrix})_{\overline{\psi}}\in(\begin{smallmatrix}^{\bot}\mathsf{Smd}(\mathcal {X})\\^{\bot}\mathsf{Smd}(\mathcal {Y})\end{smallmatrix})$.

Since $\mathcal{X}$ is special preenveloping, we have an exact sequence $0\rightarrow \overline{M}\rightarrow X\rightarrow K\rightarrow0$ in $\mathcal{A}$ with $X\in\mathcal{X}$ and $K\in^\bot\mathcal{X}$. It follows from $^\bot\mathcal{X}=^\bot\mathsf{Smd}(\mathcal{X})$ that $X\cong\overline{M}\oplus K$.
Similarly, one can show that there exists an exact sequence $0\rightarrow \textrm{ker}\overline{\varphi}\rightarrow Y\rightarrow U\rightarrow0$ in $\mathcal{B}$ with $ U\oplus{\textrm{ker}\overline{\varphi}}\cong Y$ and $U\in^\bot\mathcal{Y}$ by noting that ${\textrm{ker}\overline{\varphi}}\in{\mathsf{Smd}(\mathcal {Y})}$ and $\mathcal{Y}$ is special preenveloping. For $TK\in\mathcal{B}$, there is an exact sequence $0\rightarrow \overline{Y}{\rightarrow} N\rightarrow\stackrel{i} TK\rightarrow0$ in $\mathcal{B}$ with $N\in^\bot\mathsf{Smd}(\mathcal{Y})=^\bot\mathcal{Y}$ and $\overline{Y}\in\mathsf{Smd}(\mathcal{Y})$. Moreover, there is an exact sequence $0\rightarrow \overline{Y}\rightarrow Y'\rightarrow L\rightarrow0$ in $\mathcal{B}$ with $Y'\in\mathcal{Y}$ and $L\in^\bot\mathcal{Y}$, and so $Y'\cong\overline{Y}\oplus L$. Therefore, we have an exact sequence
$$0\rightarrow\left(\hspace{-0.15cm}\begin{array}{c} A\\B\\\end{array}\hspace{-0.15cm}\right){\rightarrow} \left(\hspace{-0.15cm}\begin{array}{c} E\\H\\\end{array}\hspace{-0.15cm}\right)_{\varphi}\rightarrow
\left(\hspace{-0.15cm}\begin{array}{c} C\\D\\\end{array}\hspace{-0.15cm}\right)_{\psi}\rightarrow0$$
in $(\mathcal{B}\downarrow T)$ with $E\cong\overline{M}\oplus K, C\cong\overline{C}\oplus K, H\cong\overline{Y}\oplus U\oplus N \oplus L, D\cong \overline{D}\oplus U\oplus N\oplus L$, $\psi=\left(\begin{smallmatrix}\overline{\psi}&0&0&0\\0&0&i&0\\ \end{smallmatrix}\right)$ and $\varphi=\left(\begin{smallmatrix}\overline{\varphi}&0&0&0\\0&0&i&0\\\end{smallmatrix}\right)$. Clearly, $\left(\begin{smallmatrix}C\\D\end{smallmatrix}\right)_{\psi}\in(\begin{smallmatrix}^{\bot}\mathcal {X}\\^{\bot}\mathcal {Y}\end{smallmatrix})$. Notice that we have an exact sequence
$0\rightarrow(\begin{smallmatrix}0\\{Y\bigoplus Y'}\end{smallmatrix}){\rightarrow} \left(\begin{smallmatrix}X\\{H}\end{smallmatrix}\right)_{\varphi}{\rightarrow}\left(\begin{smallmatrix} X\\ TX\end{smallmatrix}\right)\rightarrow0$ in $(\mathcal{B}\downarrow T)$. Thus $\left(\begin{smallmatrix}M\\Y\end{smallmatrix}\right)_{\varphi}\in\langle{\bf h}(\mathcal{X},\mathcal{Y})\rangle$. So we obtain that $\langle{\bf h}(\mathcal{X},~\mathcal{Y})\rangle$ is special preenveloping.

Conversely, assume that $\langle{\bf h}(\mathcal {X},~\mathcal {Y})\rangle$ is special preenveloping.
For any $\left(\begin{smallmatrix}A\\B\end{smallmatrix}\right)\in(\mathcal{B}\downarrow T)$, there is an exact sequence $\xi: 0\rightarrow\left(\begin{smallmatrix}A\\B\end{smallmatrix}\right){\rightarrow} \left(\begin{smallmatrix}M\\{Y}\end{smallmatrix}\right)_\varphi\rightarrow\left(\begin{smallmatrix}C\\D\end{smallmatrix}\right)
_\psi\rightarrow0$ in $(\mathcal{B}\downarrow T)$ with
$\left(\begin{smallmatrix}M\\Y\end{smallmatrix}\right)\in\langle{\bf h}(\mathcal {X},~\mathcal {Y})\rangle$ and $\left(\begin{smallmatrix}C\\D\end{smallmatrix}\right)\in^\bot\langle{\bf h}(\mathcal {X},~\mathcal {Y})\rangle$. By \prpref{Xexact}, $^\bot\langle{\bf h}(\mathcal {X},~\mathcal {Y})\rangle=(\begin{smallmatrix}^{\bot}\mathcal {X}\\^{\bot}\mathcal {Y}\end{smallmatrix})$. So the exact sequence $0\rightarrow A\rightarrow M\rightarrow C\rightarrow0$ implies that $\mathcal{X}$ is special preenveloping.  Moreover, let $A=0$ for the sequence $\xi$, then we have an exact commutative diagram
$$\xymatrix{
&0\ar[d]&0\ar[d]&0\ar[d]&\\
0\ar[r]&B\ar[r]\ar@{=}[d]&\textrm{ker}\varphi\ar[r]\ar[d]&\textrm{ker}\psi\ar[r]\ar[d]&0\\
0\ar[r]&B\ar[r]^{\simeq}\ar[d]&Y\ar[d]^\varphi\ar[r]&D\ar[d]^{\psi}\ar[r]&0\\
0\ar[r]&0\ar[r]&TM\ar[r]\ar[d]&TC\ar[r]\ar[d]&0\\
&&0&0&\\}$$
Note that $M\in\mathcal{X}\cap\,^{\bot}\mathcal{X}, \textrm{ker}\varphi\in\mathcal{Y}$, so the middle column is split. This implies that the right column is split. If follows that  $\textrm{ker}\phi\in\,^{\bot}\mathcal{Y}$ since $D\in\,^{\bot}\mathcal{Y}$. Thus the exact sequence $0\rightarrow B\rightarrow\textrm{ker}\varphi\rightarrow \textrm{ker}\phi\rightarrow0$ yields the special $\mathcal{Y}$-preenvelop of $B$.

(2)Assume that $T:\mathcal{A}\rightarrow\mathcal{B}$ is a cocompatible functor. Note that $^{\bot}\mathcal{GI}_{\mathcal{A}}\cap{\mathcal{GI}_{\mathcal{A}}}$ is the class of injective objects in $\mathcal{A}$. It follows from  (C2) in \dfnref{cocompatible} that $T(^{\bot}\mathcal{GI}_{\mathcal{A}}\cap {\mathcal{GI}_{\mathcal{A}}})\subseteq{{^{\bot}\mathcal{GI}_{\mathcal{B}}}}$. By (1) and \prpref{compatible}, it suffices to show that $T$ is $\mathcal{GI}_{\mathcal{A}}$-exact.

In fact, for each exact sequence $0\ra X \ra A\ra B\ra 0$  in $\mathcal{A}$ with $X\in{\mathcal{GI}_{\mathcal{A}}}$, there is an exact sequence $0\ra X\ra E_0\ra K_{-1}\ra 0$ in $\mathcal{A}$ such that $K_{-1}$ is a Gorenstein object and $E_0$ is a injective object. We choose an exact sequence $0\ra B\ra J_0\ra C_{-1}\ra 0$ in $\mathcal{A}$ with $J_0$ an injective object. Thus we have the following exact commutative diagram:

$$\xymatrix{&0\ar[d]&0\ar[d]&0\ar[d]&\\
0\ar[r]&X\ar[r]\ar[d]&A\ar[r]\ar[d]&B\ar[r]\ar[d]&0\\
0\ar[r]&E_0\ar[r]\ar[d]&W_{0}\ar[r]\ar[d]&J_0\ar[r]\ar[d]&0\\
0\ar[r]&K_{-1}\ar[r]\ar[d]&Z_{-1}\ar[r]\ar[d]&C_{-1} \ar[r]\ar[d]&0\\
&0&0&0&}$$
where the second row is split. Note that $0\ra TX\ra TE_0\ra TK_{-1}\ra 0$ is an exact sequence in $\mathcal{B}$ by (C1) in \dfnref{cocompatible}. Applying the functor $T$ to the above diagram, we have the following commatative diagram with exact rows and columns:

$$\xymatrix{&0\ar[d]&0\ar[d]&0\ar[d]\\
0\ar[r]&TX\ar[r]\ar[d]&TA\ar[r]\ar[d]&TB\ar[d]\\
0\ar[r]&TE_0\ar[r]\ar[d]&TW_{0}\ar[r]\ar[d]&TJ_0\ar[d]\\
0\ar[r]&TK_{-1}\ar[r]\ar[d]&TZ_{-1}\ar[r]&TC_{-1}\\
&0\\}$$
By the Snake Lemma, we have that $0\rightarrow TX\rightarrow TA\rightarrow TB\rightarrow 0$ is an exact sequence in $\mathcal{B}$.
\end{proof}
Suppose that $\Lambda=\left(\begin{smallmatrix}S & 0 \\M & R \\ \end{smallmatrix}\right)$ is a triangular matrix ring and $_{R}M_{S}$ an $R$-$S$-bimodule. Define $T\cong{\rm Hom}_R(M,~-):$ $R$-$\rm{Mod}\rightarrow$$S$-$\rm{Mod}$, then we know that $\Lambda$-$\rm{Mod}$ is equivalent to the comma category ($S$-$\rm{Mod}$$\downarrow T)$. By Proposition \ref{00} and Theorem \ref{thm:special preenveloping}(2), we have the following result.
\begin{cor}\label{corrr}
Let $\Lambda=\left(\begin{smallmatrix}S & 0 \\M & R \\ \end{smallmatrix}\right)$ be a triangular matrix ring. Assume that ${\rm pd}_{R}M$ and ${\rm fd}M_{S}$ are finite, then $\mathcal{GI}({\Lambda})$ is special preenveloping in $\Lambda$-$\rm{Mod}$ if and only if  $\mathcal{GI}({R})$ and $\mathcal{GI}({S})$ are special preenveloping in $R$-$\rm{Mod}$ and $S$-$\rm{Mod}$, respectively.
\end{cor}

Next, we provide an instance to illustrate the above result. We assume that the composition of arrows $\alpha$ and $\beta$ with $\mathfrak{t}(\alpha)=\mathfrak{s}(\beta)$ in this example, where $\mathfrak{s}$ and $\mathfrak{t}$ are two functions defined on quiver $\Q = (\Q_0, \Q_1, \mathfrak{s}, \mathfrak{t})$ sending any arrow to its starting point and ending point, respectively.

\begin{exm} \rm \label{exp}
Let $k$ be a field and $S=k\Q/\I$ a finite-dimensional $k$-algebra given by $\mathcal{Q}=$
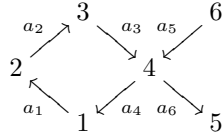
\[
\xymatrix @R=0.3cm@C=0.5cm{
  & 3\ar[dr]^{a_3} && 6 \ar[dl]_{a_5}&\\
    2\ar[ur]^{a_2} && 4 \ar[dl]^{a_4}\ar[rd]_{a_6}&\\
  & 1\ar[ul]^{a_1} && 5 &
}\]
and $\I= \langle a_1a_2, a_2a_3, a_3a_4, a_4a_1, a_5a_6\rangle$.
Let $R=k$ be the $k$-module with $k$-dimension one and $\Lambda=\left(\begin{smallmatrix} k & 0 \\ {_kS(5)_S} & S \\ \end{smallmatrix}\right)$ ($_kS(5)_S = S(5)_S$ is the simple module in $\modcat\dashed S$ corresponding to the vertex $5 \in \Q_0$).
Then $\Lambda$ is a finite-dimensional $k$-algebra, and we obtain that $\Lambda = k\Q'/\I'$ by
\[\dim_k\varepsilon_i' (\mathrm{rad}\Lambda/\mathrm{rad}^2\Lambda) \varepsilon_j' = \sharp\{a:j\to i \mid a\in \Q'_1\}\]
\begin{center}
  ($\varepsilon_i$ is the primitive idempotent of $\Lambda$ corresponding to $i\in\Q_0$,

  and $\sharp N$ the number of all elements of the set $N$),
\end{center}
where
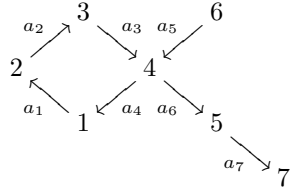
\[
\xymatrix @R=0.3cm@C=0.5cm{
  & 3\ar[dr]^{a_3} && 6 \ar[dl]_{a_5}&\\
    2\ar[ur]^{a_2} && 4 \ar[dl]^{a_4}\ar[rd]_{a_6}&\\
  & 1\ar[ul]^{a_1} && 5 \ar[rd]_{a_7} & \\
  & & & & 7
}\]
and $\I' = \langle a_1a_2, a_2a_3, a_3a_4, a_4a_1, a_5a_6, a_6a_7 \rangle$, see \cite[Chap II, Definition 3.1 and Theorem 3.7]{ASS2006}.
Then the following statements hold:
\begin{itemize}
  \item  $T:=\Hom{k}{S(5)}{-}$ is cocompatible by Proposition \ref{00} because $S(5)$ is a left projective $k$-module and right flat $S$-module, that is, ${\rm{pd}}_{k}S(5)=0$ and ${\rm{fd}}S(5)_{S}=0$.
  \item The Gorenstein injective left modules over $S$ and $\Lambda$ can be calculated by the dual of \cite[Theorem 4.1]{MR3869172}
    (or the dual of \cite[Theorem 2.5]{K2014}) because $S$ and $\Lambda$ are gentle algebras, to be more precise,
    \begin{center}
      $\mathcal{GI}({S}) = \left\{{_SS(1)}, {_SS(2)},{_SS(3)},{_S\big(\begin{smallmatrix} 4\\ 5 \end{smallmatrix}\big)}\right\}
      \cup \{{_SE(i)} \mid i\in \Q_0\}$;

      $\mathcal{GI}({\Lambda}) = \left\{{_{\Lambda}S(1)}, {_{\Lambda}S(2)},{_{\Lambda}S(3)},{_{\Lambda}\big(\begin{smallmatrix} 4\\ 5 \end{smallmatrix}\big)}\right\}
      \cup \{{_{\Lambda}E(i)} \mid i\in \Q_0'\}$.
    \end{center}
  \item $\mathcal{GI}({R}) \ ( = \mathcal{GI}({k}) )\ =k$.
\end{itemize}
One can check that $\mathcal{GI}({R})$ and $\mathcal{GI}({S})$ are special preenveloping in $R\dashed\modcat$ and $S\dashed\modcat$, respectively.
Then $\mathcal{GI}({\Lambda})$ is special preenveloping in $\Lambda\dashed\rm{mod}$ by Corallary \ref{corrr}, and vice verse.
In fact, we can show the above conclusion by the Auslander-Reiten quiver $\Gamma(\Lambda\dashed\modcat)$ of $\Lambda\dashed\modcat$ which is shown in Figure \ref{AR-quiver}
(The modules enclosed with {\color{blue}solid lines} and {\color{orange}dotted lines} are
indecomposable injective and non-injective Gorenstein-injective modules, respectively).
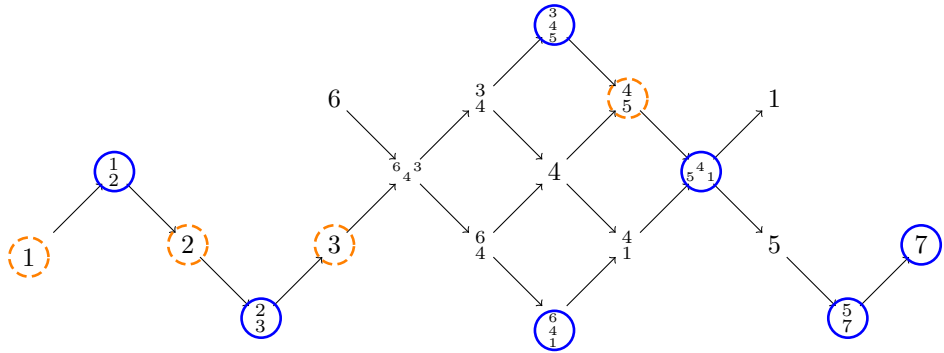
\begin{figure}[htbp]
\centering
\begin{tikzpicture} [scale=0.65]
\draw [->] (0.5,0.5) -- (1.5,1.5);
\draw [->] (2,1.5) -- (3,0.5);
\draw [->] (3.5,0) -- (4.5,-1);
\draw [->] (5,-1) -- (6,0);
\draw [->] (6.5,0.5) -- (7.5,1.5);
\draw [->] (6.5,3) -- (7.5,2);
\draw [->] (8,2) -- (9,3);
\draw [->] (8,1.5) -- (9,0.5);
\draw [->] (9.5,0) -- (10.5,-1);
\draw [->] (9.5,3.5) -- (10.5,4.5);
\draw [->] (9.5,3) -- (10.5,2);
\draw [->] (9.5,0.5) -- (10.5,1.5);
\draw [->] (11,4.5) -- (12,3.5);
\draw [->] (11,2) -- (12,3);
\draw [->] (11,1.5) -- (12,0.5);
\draw [->] (11,-1) -- (12,0);
\draw [->] (12.5,3) -- (13.5,2);
\draw [->] (12.5,0.5) -- (13.5,1.5);
\draw [->] (14,2) -- (15,3);
\draw [->] (14,1.5) -- (15,0.5);
\draw [->] (15.5,0) -- (16.5,-1);
\draw [->] (17,-1) -- (18,0);
\draw ( 0   , 0   ) node{$1$};                  \draw[orange][line width=1pt] (0   , 0   ) circle (0.4) [dash pattern=on 4pt off 2pt];
\draw ( 1.75, 1.75) node{${^1_2}$};             \draw[  blue][line width=1pt] (1.75, 1.75) circle (0.4);
\draw ( 3.25, 0.25) node{$2$};                  \draw[orange][line width=1pt] (3.25, 0.25) circle (0.4) [dash pattern=on 4pt off 2pt];
\draw ( 4.75,-1.25) node{${^2_3}$};             \draw[  blue][line width=1pt] (4.75,-1.25) circle (0.4);
\draw ( 6.25, 3.25) node{$6$};
\draw ( 6.25, 0.25) node{$3$};                  \draw[orange][line width=1pt] (6.25, 0.25) circle (0.4) [dash pattern=on 4pt off 2pt];
\draw ( 7.75, 1.75) node{\tiny ${^6}{_4}{^3}$};
\draw ( 9.25, 3.25) node{${^3_4}$};
\draw ( 9.25, 0.25) node{${^6_4}$};
\draw (10.75, 4.75) node{${^{_3}_{^4_5}}$};     \draw[  blue][line width=1pt] (10.75, 4.75) circle (0.4);
\draw (10.75, 1.75) node{$4$};
\draw (10.75,-1.5 ) node{${^{_6}_{^4_1}}$};     \draw[  blue][line width=1pt] (10.75,-1.5 ) circle (0.4);
\draw (12.25, 3.25) node{${^4_5}$};             \draw[orange][line width=1pt] (12.25, 3.25) circle (0.4) [dash pattern=on 4pt off 2pt];
\draw (12.25, 0.25) node{${^4_1}$};
\draw (13.75, 1.75) node{\tiny ${_5}{^4}{_1}$}; \draw[  blue][line width=1pt] (13.75, 1.75) circle (0.4);
\draw (15.25, 3.25) node{$1$};
\draw (15.25, 0.25) node{$5$};
\draw (16.75,-1.25) node{${^5_7}$};             \draw[  blue][line width=1pt] (16.75,-1.25) circle (0.4);
\draw (18.25, 0.25) node{${7}$};                \draw[  blue][line width=1pt] (18.25, 0.25) circle (0.4);
\end{tikzpicture}
\caption{The Auslander-Reiten quiver of $\Lambda$ in Example \ref{exp}}
\label{AR-quiver}
\end{figure}
Notice that the composition of two arrows $\alpha$ and $\beta$ we defined is $\alpha\beta$ ($\mathfrak{t}(\alpha)=\mathfrak{s}(\beta)$),
then the representations of indecomposable left $\Lambda$-modules and the Auslander-Reiten quiver $\Gamma(\Lambda\dashed\modcat)$ of $\Lambda\dashed\modcat$ are opposite to that of right $\Lambda$-modules and $\modcat\dashed\Lambda$, respectively.
\end{exm}


\section*{Acknowledgments}
We would like to thank Yu-Zhe Liu for his helpful discussions.

\def\romsup#1{{\edef\next{\the\font}$^{\next#1}$}}
\providecommand{\bysame}{\leavevmode\hbox to3em{\hrulefill}\thinspace}
\providecommand{\MR}{\relax\ifhmode\unskip\space\fi MR }
\providecommand{\MRhref}[2]{%
  \href{http://www.ams.org/mathscinet-getitem?mr=#1}{#2}
}
\providecommand{\href}[2]{#2}

\end{document}